\documentclass[11pt,letterpaper]{amsart}
\usepackage{amssymb,xfrac}
\usepackage{amsmath}
\usepackage{breqn}
\usepackage{xcolor}
\usepackage{graphics, cite, bookmark}
\usepackage{epsfig}
\usepackage{enumerate}
\usepackage{hyperref}
\usepackage{epstopdf}
\usepackage{ esint }
\usepackage{amsthm}
\usepackage{lmodern}
\usepackage[title]{appendix}

\newtheorem{theorem}{Theorem}[section]
\newtheorem*{theorem*}{Theorem}
\newtheorem{lemma}[theorem]{Lemma}
\newtheorem{proposition}[theorem]{Proposition}

\theoremstyle{definition}

\theoremstyle{remark}
\newtheorem{remark}{Remark}[section]

\numberwithin{equation}{section}
\allowdisplaybreaks

\newcommand{\R}{\mathbb{R}}
\newcommand{\N}{\mathbb{N}}
\renewcommand{\S}{\mathbb{S}}

\newcommand{\sym}{\mathrm{sym}}
\newcommand{\Sym}{\mathrm{Sym}}

\newcommand{\sdel}{\delta^{\sfrac{1}{2}}}
\usepackage{bbm}

\begin{document}

\title
[]{Flexibility of Codimension One $C^{1,\theta}$ Isometric Immersions}

\author{Dominik Inauen}

\address{Dominik Inauen, Universit\'e de Fribourg, CH-1700, Fribourg, Switzerland.  E-mail:{\tt dominik.inauen@unifr.ch}}

\begin{abstract}
We study the problem of constructing $C^{1,\theta}$ isometric immersions of Riemannian metrics on $n$-dimensional domains into $\R^{n+1}$. While the classical Nash--Kuiper theorem established the flexibility of $C^1$ isometries, subsequent work has extended this to $C^{1,\theta}$ isometries for certain $\theta$, though the optimal exponent remains unknown. In this work we show that any short immersion can be uniformly approximated by $C^{1,\theta}$ isometric immersions for $\theta< 1/(1+2(n-1))$,  improving upon the previously known exponent for $n\geq 3$. The improvement is obtained via a convex integration scheme incorporating a refined iterative integration by parts procedure resting on a detailed structural analysis of error terms and the interaction of multiple frequency scales.
\end{abstract}

\date{\today}

\maketitle 

\section{Introduction}
Let $(M,g)$ be a smooth $n$-dimensional Riemannian manifold and $m\in \N$. An isometric immersion $u:M\to \R^m$ is a continuously differentiable map $u$ such that the induced metric $u^\sharp e$ equals $g$, where $e $ denotes the Euclidean metric on $\R^m$.  In local coordinates on $\Omega\subset\R^n$, this condition reads
\begin{equation}\label{e:isometry}
    Du^TDu = g\,,
\end{equation}
where $Du$ is the Jacobian of the map $u:\Omega\to \R^m$, and $g$ is identified with the positive definite symmetric matrix field of coefficients representing the metric in these local coordinates.

Equation \eqref{e:isometry} is a system of $n_*:= n(n+1)/2$ non-linear partial differential equations in $m$ unknowns. For low codimension $m=n+1$ and $n\geq 3$ it is overdetermined, having $n_*>m$ equations. Nevertheless, the groundbreaking Nash-Kuiper theorem \cite{Nash54,Kuiper} established a surprising abundance of solutions to \eqref{e:isometry}: in any $C^0$-neighborhood of any \emph{short} immersion $\underline u$ (meaning $\underline u^\sharp e \leq g$ in the sense of quadratic forms), there exists an isometric immersion $u$ of regularity $C^1$. 

This so-called flexibility contrasts classical  rigidity results for smooth isometric immersions. A prominent example
is the rigidity theorem of the Weyl problem: if $g$ is any metric of positive Gaussian curvature on $\S^2$, then any two isometric immersions $u,v:\mathbb S^2\to \R^3$ of class $C^2$ differ only by  a rigid motion, a result due to  Cohn-Vossen \cite{CohnVossen1927} and Herglotz \cite{Herg1943} (see also \cite{Pogo1951} for a more general result, implying the above). 

This dichotomy of flexibility at low regularity and rigidity at high regularity motivates the following question: is there a critical H\"older threshold  $\theta_0\in ]0,1[$ separating the two regimes? More precisely, does there exist $\theta_0$ such that 
\begin{enumerate}[(i)]
\item isometric immersions $u\in C^{1,\theta}(M,\R^{n+1})$ exhibit (some sort of) rigidity if $\theta>\theta_0$,
\item the Nash--Kuiper theorem extends to $C^{1,\theta}$ for all $\theta<\theta_0$?
\end{enumerate}

 Analogous phenomena occur in other geometric PDE, such as the Monge--Amp\`ere equation, where sufficiently regular solutions with positive right hand side are rigid (convex), whereas very weak $C^{1,\theta}$ solutions exhibit flexibility for low $\theta$, see \cite{LewickaPakzad,MartaSystems,Martasystems2,Martasystems3,IL25,ILfull}. A particularly well-studied parallel arises in fluid dynamics: Onsager's conjecture for the incompressible Euler equations asserts that $C^{0,\theta}$ weak solutions conserve energy (rigidity) for $\theta>1/3$ , while many non-conservative solutions exist for $\theta < 1/3$ (flexibility). The rigidity was established in \cite{CET1994} and flexibility in \cite{Isett2018} (see also \cite{BDSV2019}) following \cite{DeSz2012,DS14,BDIS2015, BDS16}.

Both the existence and value of such a threshold  $\theta_0$ for isometric immersions remain unknown. Parallels with Onsager's conjecture, in particular the shared convex integration methodology, suggest $\theta_0=1/3$ as a potential threshold. 
On the other hand, \cite{DI2020,CaoIn2024} demonstrate that $C^{1,\sfrac{1}{2}}$ is a critical space \emph{in a suitable sense}, suggesting $\theta_0 = 1/2$ (see also \cite[Question 36]{Gromov2017}, \cite[Section 10]{CShighdim}, \cite{InauenMenon},\cite{CHIDarboux}). 

\textbf{$C^{1,\theta}$ isometries: known results.} The first results on isometric immersions of intermediate regularity $C^{1,\theta}$ were established by Yu. F. Borisov in the mid-20th century. He proved in \cite{Borisov1958,Borisov1958ii,Borisov1959} that the rigidity theorem of Cohn-Vossen and Herglotz remains true for isometries of class $C^{1,\theta}$ for $\theta >2/3$ (see also \cite{CDS} for a short proof, and \cite{Reza}). 

On the other hand, he announced in \cite{Borisov1965} that the Nash-Kuiper theorem remains true for isometric immersions of open, bounded $n$-dimensional sets of regularity $C^{1,\theta}$ for 
$\theta< 1/(1+n(n+1))$,
the Borisov-exponent, with a potential improvement to $\theta<1/5$ for $n=2$. These claims were confirmed in \cite{CDS} and \cite{DIS} and extended to compact manifolds in \cite{CaoSze2022}. 

Recently, \cite{CaoHirschInauen25} improved the exponent to $\theta< 1/(1+n(n-1))$,
achieving the Onsager threshold $\theta<1/3$ when $n=2$.

\textbf{Higher codimension.}
For $m>n+1$, the flexibility regime extends considerably. Existence results for smooth isometries are due to Nash \cite{Nash1956}, with codimension bounds improved by Gromov and G\"unther \cite{GromovPdr, Guenther}. Results for $C^{1,\theta}$ isometries in higher codimension can be found in \cite{Kaellen,DI2020,CaoIn2024,CaoSz2023,Zhiwen,MartaFull}.  

\textbf{Main result.}
In this paper, we focus on the most constrained case of codimension one, and improve the exponent of \cite{CaoHirschInauen25} when $n\geq 3$. More precisely, we have
\begin{theorem}\label{t:main}
    Let $n\geq2$ and let $g\in C^2(\bar \Omega,\Sym_n^+)$ be a Riemannian metric on a smooth, bounded open set $\Omega\subset\R^n$. Then for any short immersion $\underline u\in C^1(\bar\Omega,\R^{n+1})$, any $\varepsilon>0$ and any H\"older exponent
    \begin{equation}\label{e:ourexponent}
        \theta< \frac{1}{1+2(n-1)}\,,
    \end{equation} there exists an isometric immersion $u\in C^{1,\theta}(\bar\Omega,\R^{n+1}) $ such that 
    \begin{equation}\label{e:c0closedness}
        \|u-\underline u\|_0 < \varepsilon.
    \end{equation}
\end{theorem}
\begin{remark}
As in \cite{CaoHirschInauen25}, Theorem~\ref{t:main} extends to the case of embeddings by standard arguments. 
Moreover, combining the present construction with the global framework developed in \cite{CaoSze2022}, the result carries over to compact manifolds. 
Finally, the same mechanism applies in the setting of the very weak Monge–Amp\`ere equation, yielding an analogous improvement of the H\"older exponent.
\end{remark}

The improvement relies on a refined integration-by-parts mechanism as compared to \cite{CaoHirschInauen25}, which is based on a more precise structural analysis of the error terms arising in the convex integration scheme. 
 In what follows, we provide a heuristic sketch of the proof: first recalling the classical Nash iteration (Section~\ref{ss:Nash}), then explaining $C^{1,\theta}$ convergence (Section~\ref{ss:introconvergence}), reviewing the iterative integration by parts of \cite{CaoHirschInauen25}, and finally presenting the refinement leading to Theorem~\ref{t:main}. Since the details of the construction are very technical, the discussion is heuristic.
 
\subsection{The Nash iteration}\label{ss:Nash}

Following Nash \cite{Nash54}, we construct the isometric immersion $u$  as the limit of an iterative scheme, the  Nash iteration. 
Starting from the initial short map $\underline u$, we build a sequence of smooth short immersions $\{u_q\}_{q\in\mathbb N}$ converging in $C^{1,\theta}$ and such that the \emph{metric deficit} $
g - Du_q^T Du_q$
decays exponentially:
\begin{equation}
    \|g-Du_q^T Du_q\|_0 
    \le K^{-1}\|g-Du_{q-1}^T Du_{q-1}\|_0
\end{equation}
for some large constant $K\ge 1$. 
In fact, for technical reasons, we enforce a double-exponential decay of the defect.

Given $u_{q-1}$, the construction of $u_q$, called a \emph{stage}, begins by decomposing the positive definite defect into a finite sum of so-called primitive metrics:
\begin{equation}\label{e:primintrodecomp}
    g-Du_{q-1}^T Du_{q-1} 
    = \sum_{i=1}^N a_i^2\, \eta_i\otimes\eta_i,
\end{equation}
where the coefficients $a_i$ are smooth and $\eta_i\in\mathbb S^{n-1}$. An elementary linear algebra argument (see Lemma \ref{l:basicdecomp}) shows that if the defect is sufficiently close to a constant positive definite matrix (something that can be arranged along the sequence by translation and rescaling), then the decomposition holds with $N=n_*=n(n+1)/2$.

We then add $N$ corrugations $w_{q,1},\dots,w_{q,N}$ sequentially:
\[
u_{q,0}=u_{q-1}, 
\qquad 
u_{q,i}=u_{q,i-1}+w_{q,i}, 
\quad i=1,\dots,N,
\]
and set $u_q=u_{q,N}$.

Each perturbation $w_{q,i}$ is designed to increase the induced metric by the corresponding primitive metric:
\begin{equation}\label{e:introinduced}
Du_{q,i}^T Du_{q,i}
=
Du_{q,i-1}^T Du_{q,i-1}
+
a_i^2 \eta_i\otimes\eta_i
+
E_i,
\end{equation}
up to an error term $E_i$ satisfying
\begin{equation}\label{e:introprimerror}
\|E_i\|_0
\le (NK)^{-1}
\|g-Du_{q-1}^T Du_{q-1}\|_0,
\end{equation}
which yields \eqref{e:primintrodecomp} by a summation over $i$ .

The basic ansatz for a perturbation achieving \eqref{e:introinduced}–\eqref{e:introprimerror} is
\begin{equation}\label{d:basicpert}
w_{q,i}
=
\frac{a_i}{\nu_i}
\left(
\gamma_1(\nu_i x\cdot\eta_i)\tau_i
+
\gamma_2(\nu_i x\cdot\eta_i)\zeta_i
\right),
\end{equation}
where $\nu_1\le\dots\le\nu_N$ are large frequency parameters, 
$\gamma_1,\gamma_2\in C^\infty(\mathbb S^1)$ are oscillatory functions, 
and $\tau_i$ and $\zeta_i$ are tangential and normal vector fields to $u_{q,i-1}$, respectively.
From \eqref{d:basicpert} and \eqref{e:primintrodecomp} one sees that
\[
\|Du_q-Du_{q-1}\|_0
\le
C\|g-Du_{q-1}^T Du_{q-1}\|_0^{1/2}
+
\frac{C}{\nu_1},
\]
so the sequence is Cauchy in $C^1$ provided the frequencies are chosen large enough.

\subsection{$C^{1,\theta}$ convergence}\label{ss:introconvergence} We now explain how to obtain a sequence converging in $C^{1,\theta}$. A  systematic choice of frequency parameters $\nu_i$ is crucial. To see why, set \[\delta_q =K^{-q}\|g-Du_0^TDu_0\|_0\] so that the sequence $\{u_q\}_q$ of the previous subsection satisfies 
\[ \|g-Du_q^TDu_q\|_0\leq  \delta_q\,.\] 
Now observe that if $\nu_i$ are  large, then $\|D^2u_q\|_0 \leq C\delta_{q-1}^{1/2} \nu_N$, since the leading  contribution to $D^2u_q$ arises when both derivatives fall on the rapidly oscillating perturbation. Thus, controlling the growth of frequencies controls the blow-up of the second derivative.

This requires an analysis of the error terms $E_i$. An estimate disregarding the specific structure of $E_i$ yields 
\[\|E_i\|_0\leq C\delta_{q-1}\frac{\nu_{i-1}}{\nu_i}\,,\]
with $\nu_0 = \|D^2u_{q-1}\|_0/\delta_{q-1}^{1/2}$.
In order to satisfy \eqref{e:introprimerror}, one is therefore forced to choose an increasing sequence $\nu_i = K\nu_{i-1}$, resulting in the exponential blowup 
\[\|D^2u_q\|_0\leq C\|D^2u_{q-1}\|_0K^N\leq C(K^N)^q\,.\]
Since however $\|Du_q-Du_{q-1}\|_0\leq C\delta_{q-1}^{1/2} = C K^{-(q-1)/2}$, an interpolation between the $C^1$ and $C^2$ norms shows that this sequence is Cauchy in $C^{1,\theta}$ for all $\theta< 1/(1+2N)$. Taking $N=n_*$ yields the Borisov exponent. In \cite{DIS},  a diagonalization of the defect (available in two dimensions via isothermal coordinates) reduces $N$ from $3$ to $2$, thereby explaining the exponent $1/5.$

\subsection{Iterative integration by parts}\label{ss:introibp}

In \cite{CaoHirschInauen25}, a better regularity is achieved by an iterative integration by parts process. 
The key observation is that all (except one) leading-order error terms in $E_i$ are of the form
\[
\gamma(\nu_i x\cdot\eta_i)\, M,
\]
where $\gamma\in C^\infty(\mathbb S^1)$ has mean zero, and $M$ is a symmetric matrix oscillating at a lower frequency $\mu\le \nu_{i-1}$.

Any symmetric matrix $M$  can be decomposed as \begin{equation}\label{e:primsplit}
M = \sym(\alpha(M)\otimes\eta_i) + \pi(M),
\end{equation}
for some vector $\alpha(M)\in\mathbb R^n$ and a remainder $\pi(M)$. 
Writing $\gamma^1$ for the mean-zero primitive of $\gamma$, one obtains
\[
\gamma(\nu_i x\cdot\eta_i) M = \sym\left(D\left(\frac{\gamma^1(\nu_i x\cdot\eta_i)}{\nu_i}\alpha(M)\right)\right)-\frac{\gamma^1(\nu_i x\cdot\eta_i)}{\nu_i}\sym(D\alpha(M))+\gamma\,\pi(M).
\]
This is referred to in \cite{CaoHirschInauen25} as an integration by parts. 
Since $M$ oscillates at frequency $\mu$, the second term on the right-hand side is now of smaller size $\|M\|_0\, \mu/\nu_i$. Iterating this procedure $J$ times yields
\[\gamma(\nu_i x\cdot\eta_i) M=
\sym(D w_c) + \mathcal E + \mathcal F,
\]
where $w_c$ is a corrector vector field, 
\[
\|\mathcal E\|_0 \le C \Big(\frac{\nu_{i-1}}{\nu_i}\Big)^J \|M\|_0,
\]
and $\mathcal F$ is a remainder of size comparable to $\|M\|_0$.

In \cite{CaoHirschInauen25}, a modified perturbation (different from \eqref{d:basicpert}) is chosen so that an additional symmetric gradient appears in the induced metric:
\[
Du_{q,i}^T Du_{q,i}
=
Du_{q,i-1}^T Du_{q,i-1}
+
a_i^2\eta_i\otimes\eta_i
+
\sym(Dw)
+
E_i.
\]
Choosing $w=-w_c$ cancels the leading order error term $\gamma(\nu_i x\cdot\eta_i)M$ up to the much smaller error $\mathcal E$ and the remainder $\mathcal F$. 
Applying this cancellation to all leading error terms in $E_i$ shows that, disregarding $\mathcal F$, condition \eqref{e:introprimerror} can be achieved by choosing
\[
\nu_i = K^{1/J}\nu_{i-1}.
\]
Since $J$ can be taken arbitrarily large, this allows the frequencies to grow arbitrarily slowly, that is, $\nu_i\approx \nu_{i-1}$.

The remaining difficulty lies in the presence of the large remainder $\mathcal F$.  However, if $\eta_i\cdot e_1\neq 0$, one checks that $\pi(M)$, and hence $\mathcal F$, is a symmetric matrix supported in the bottom-right $(n-1)\times(n-1)$ block. 
In \cite{CaoHirschInauen25}, the directions $\eta_i$ are therefore ordered so that $\eta_i\cdot e_1\neq 0$ for $i=1,\dots,n$, and such that 
\[
\{\eta_j\otimes\eta_j : j=n+1,\dots,n_*\}
\]
spans the space of symmetric matrices supported in the bottom-right $(n-1)\times(n-1)$ block.

With this ordering, the first $n$ perturbations can be corrected using suitable $w_c$, producing an error of size $\sum_{i=1}^n(\nu_{i-1}/\nu_i)^J$ and a remainder of size $\sum_{i=1}^n \nu_{i-1}/\nu_i$ lying in the span of $\{\eta_j\otimes\eta_j : j=n+1,\dots,n_*\}$. 
This remaining term can then be canceled exactly by adjusting the coefficients of the perturbations $j=n+1,\dots,n_*$.

Consequently, the first $n$ frequencies can be chosen essentially constant,
\[
\nu_n \approx \dots \approx \nu_1 \approx \nu_0
\]
(for large $J$), whereas for $i=n+1,\dots,n_*$ one must use the standard choice $\nu_i=K\nu_{i-1}$. 
This leads to
\[
\nu_n = K^{\,n_*-n+n/J}\nu_0,
\]
and therefore convergence in $C^{1,\theta}$ for all
\[
\theta < \frac{1}{1+2(n_*-n+n/J)}.
\]
Letting $J\to\infty$ yields the exponent in \cite{CaoHirschInauen25}.

\subsection{Novelties}\label{ss:intronovelties}

In \cite{CaoHirschInauen25}, we chose $w=0$ in the perturbations $j=n+1,\ldots,n_*$ for the following reason. 
When trying to apply the integration by parts process to a term $\gamma(\nu_i x\cdot \eta_i) M$ with $\eta_i\cdot e_1 = 0$, the decomposition \eqref{e:primsplit} still holds, but the large remainder $\pi(M)$ — which cannot be handled by integration by parts — no longer lies in the lower-dimensional span of primitive metrics yet to be added. 
Consequently, the full remainder cannot be canceled by adjusting the coefficients $a_j$ for $j\ge i+1$, and one must choose $\nu_i = K \nu_{i-1}$ so that the uncancelable part of the remainder is sufficiently small.

The crucial insight of this paper is that, in fact, three frequencies are present in the definition of $u_{q,i}$: the frequency $\nu_i$ of the perturbation added to $u_{q,i-1}$, the frequency $\nu_{i-1}$ at which $u_{q,i-1}$ (and hence its tangential and normal vector fields) oscillate, and the frequency $\nu_0$ at which the coefficient $a_i$ oscillates. 
To illustrate heuristically how this matters, consider the scalar setting of the very weak Monge-Ampère equation and assume
\[
u_{q,i} = u_{q,i-1} + \frac{a_i}{\nu_i}\gamma(\nu_i x\cdot \eta_i).
\]
A closer analysis of the error terms in $E_i$ reveals two types of leading-order contributions (see \eqref{e:importanterror}, \eqref{e:secondder}):
\[
\frac{a_i}{\nu_i}\gamma(\nu_i x\cdot \eta_i) D^2 u_{q,i-1}, \qquad
\frac{\gamma(\nu_i x\cdot\eta_i)}{\nu_i^2} \nabla a_i \otimes \nabla a_i.
\]
The second term is of order $\nu_0^2/\nu_i^2$ and is negligible as soon as $\nu_i \ge K^{1/2}\nu_0$. The first term is of order $\nu_{i-1}/\nu_i$, seemingly requiring $\nu_{i} = K\nu_{i-1}$. However, expanding $D^2 u_{q,i-1}$ gives
\[
D^2 u_{q,i-1} = D^2 u_{q,i-2} + \nu_{i-1} a_{i-1} \gamma''(\nu_{i-1} x\cdot \eta_{i-1}) \eta_{i-1}\otimes \eta_{i-1} + \gamma' \nabla a_{i-1} \otimes \eta_{i-1} + \frac{\gamma}{\nu_{i-1}} D^2 a_{i-1}.
\]
Disregarding the last two terms (which are of size $\leq C\nu_0$) and, for the moment, $D^2 u_{q,i-2}$, the leading-order error is (see \eqref{e:subst2_inductiveinduced})
\[
\frac{\nu_{i-1}}{\nu_i} \gamma(\nu_i x\cdot \eta_i) a_i a_{i-1} \gamma''(\nu_{i-1} x\cdot \eta_{i-1}) \eta_{i-1} \otimes \eta_{i-1} 
= \frac{\nu_{i-1}}{\nu_i} \gamma(\nu_i x\cdot \eta_i) A(\nu_{i-1} x\cdot \eta_{i-1}, x) \otimes \eta_{i-1},
\]
where we defined $A(t,x) = \gamma''(t) a_i(x) a_{i-1}(x) \eta_{i-1}$. 
Ideally, we would like a corrector $w_c$ such that
\[
\gamma(\nu_i x\cdot \eta_i) A(\nu_{i-1} x\cdot \eta_{i-1}, x) \otimes \eta_{i-1} = \sym(D w_c) + O((\nu_{i-1}/\nu_i)^J) + \mathcal F,
\]
with $\mathcal F$ in the span of $\eta_j\otimes \eta_j$, $j>i$. As explained above, this is not always possible. However, if $\eta_i$ and $\eta_{i-1}$ belong to the \emph{same subfamily}, the decomposition works up to an additional error of size $\nu_0/\nu_i$.

To explain the notion of a subfamily, let $e_i\in\R^n$ denote the standard basis vector, and define, for $1\le i\le j \le n$, the unit vectors
\begin{equation*}
\eta_{ij} = \frac{e_i + e_j}{|e_i + e_j|}.
\end{equation*}
For $i=1,\dots,n$, consider moreover the subspace
\begin{equation*}
\mathcal V_i = \mathrm{span}_\R \{\eta_{kl}\otimes \eta_{kl} : i\le k \le l \le n\},
\end{equation*}
of symmetric matrices supported in the bottom-right $(n-i+1)\times (n-i+1)$ block. In particular, $\mathcal V_1 = \Sym_n$, so that the decomposition \eqref{e:primintrodecomp} holds with $\{\eta_{ij}\}$ replacing $\{\eta_i\}$ (see Lemma \ref{l:basicdecomp}). 
If both $\eta_i$ and $\eta_{i-1}$ belong to the same subfamily $\{\eta_{kl}\}_{ l\geq k}$ for some fixed $k$, then in fact
\[
A(\nu_{i-1} x\cdot \eta_{i-1}, x) \otimes \eta_{i-1} 
= \sym(\alpha(A\otimes\eta_{i-1})\otimes \eta_i) + \pi,
\]
for a remainder $ \pi\in \mathcal V_{k+1}$ belonging to the span of primitive metrics yet to be added. Moreover, the derivative splits as
\[
D(A(\nu_{i-1}x\cdot \eta_{i-1}, x)) = \nu_{i-1}\partial_t A \otimes \eta_{i-1} + D_x A.
\]
Crucially, the leading term again has a decomposition whose remainder $\pi$ lies in the lower-dimensional space $\mathcal V_{k+1}$ (canceled later), while $D_x A$ is of size $\nu_0$ and hence small. Thus, iterating the integration by parts procedure produces a corrector $w_c$ such that
\[
\gamma(\nu_i x\cdot \eta_i) A(\nu_{i-1} x\cdot \eta_{i-1}, x) \otimes \eta_{i-1} = \sym(D w_c) + O((\nu_{i-1}/\nu_i)^J) + \mathcal F_l + \mathcal F_m,
\]
with $\mathcal F_l\in \mathcal V_{k+1}$ and $\|\mathcal F_m\|_0 \le C \nu_0 / \nu_i$, which is small enough for $\nu_i \ge K \nu_0$.

When crossing from one subfamily to the next, i.e., $\nu_{i-1}$ and $\nu_i$ do not belong to the same subfamily, $D^2 u_{i-1}$ contains terms that cannot be handled by integration by parts in direction $\eta_i$. Thus, the frequency must be increased by a factor $K$ at each family transition. 
Denoting by $\nu_{ij}$ the frequency of the perturbation corresponding to the primitive metric $\eta_{ij}\otimes\eta_{ij}$, for large $J$, the resulting sequence of frequencies is
\begin{align*}
& \nu_{1n} \approx \dots \approx \nu_{11} \approx \nu_0,\\
& \nu_{2n} \approx \dots \approx \nu_{22} = K\nu_{1n},\\
& \dots\\
& \nu_{n-1,n} \approx \nu_{n-1,n-1} = K \nu_{n-2,n},\\
& \nu_{nn} = K \nu_{n-1,n}.
\end{align*}
This yields a final frequency $\nu_n \approx K^{n-1} \nu_0$, explaining the exponent \eqref{e:ourexponent} by the argument in Section \ref{ss:introconvergence}.

\subsection{Structure of the paper} We conclude this introduction by fixing notation.  In Section \ref{s:prelim} we collect the necessary preliminaries: a lemma for the decomposition of the defect into primitive metrics, and the iterative integration by parts result, Proposition~\ref{p:IBP}, which relies on the algebraic matrix decomposition given in Lemma~\ref{l:algebraicdecomp}.In Section~\ref{s:step}, we introduce in Proposition~\ref{p:step} the form of our perturbation and analyze its effect on the induced metric. Combining this with the iterative integration by parts procedure, we prove in Section~\ref{s:stage} the main iterative proposition. Finally, in Section~\ref{s:pfofmain}, we iterate this proposition to construct the sequence $\{u_q\}$ and thereby complete the proof of Theorem~\ref{t:main}. 
For the reader’s convenience, the appendix contains the definition of H\"older spaces and their (semi-)norms, interpolation estimates, the Leibniz rule, and standard estimates for the mollification of H\"older functions.

\subsection{Notation}

For a matrix $M\in \R^{n\times n}$, we denote by $\sym(M)=\frac{1}{2}(M+M^T)$
its symmetric part, and we write $\Sym_n$ for the space of symmetric $n\times n$ matrices. For $a,b\in \R^n$, we set $
a\odot b = a\otimes b + b\otimes a = 2\,\sym(a\otimes b)$.
The symbol $C$ denotes a positive constant whose value may change from line to line. If $C$ depends on a parameter $p$, we write $C=C(p)$ or indicate the dependence explicitly.

For $1\le i\le j\le n$, we define
\begin{equation}
\label{d:etas}
\eta_{ij} = \frac{e_i+e_j}{|e_i+e_j|},
\end{equation}
where $e_i\in\R^n$ denotes the $i$-th standard basis vector. 

For $i=1,\dots,n$, we introduce the subspace
\begin{equation}
\label{d:V_i}
\mathcal V_i = \mathrm{span}_\R\{\eta_{kl}\otimes \eta_{kl} : i\le k\le l\le n\},
\end{equation}
which consists of all symmetric matrices supported in the bottom-right $(n-i+1)\times (n-i+1)$ block. We also set
\begin{equation}\label{d:V_n+1}
\mathcal V_{n+1} = \{0\},
\end{equation}
and define the positive definite matrix 
\begin{equation}\label{d:H_0}
    H_0 = \sum_{1\leq i\leq j\leq n}\eta_{ij}\otimes\eta_{ij}\,.
\end{equation}
\newpage
\section{Preliminaries: Decomposition and iterative integration by parts}\label{s:prelim}
In this section we collect auxiliary results needed for the iteration scheme. We begin with a basic algebraic decomposition of symmetric matrices in terms of the primitive metrics $\eta_{ij}\otimes\eta_{ij}$. We then recall a refined decomposition result (Lemma~\ref{l:babykaellen}), which allows one to absorb certain error terms into the primitive decomposition. For completeness, we include the elementary proof of the basic algebraic lemma, while referring to \cite{CaoHirschInauen25} for the refined argument.

Finally, in Subsection~\ref{ss:ibp}, we establish the iterative integration by parts proposition, which will serve as the analytic core of the construction.

\subsection{Decomposition}\label{ss:decomp}

\begin{lemma}\label{l:basicdecomp} Let $\{\eta_{ij}\}$ be defined as in \eqref{d:etas}. There exists $r_D= r_D(n)>0$ and linear maps $L_{ij}:\Sym_n \to \R$ such that for any symmetric matrix $M$ it holds 
\begin{equation}
    \label{e:basicdecomp} 
    M = \sum_{1\leq i\leq j\leq n } L_{ij}(M)\eta_{ij}\otimes \eta_{ij}\,, 
\end{equation}
and $L_{ij}(M) \geq r_D$ for all $M$ such that $|M-H_0 |<r_D$. 
\end{lemma}
\begin{proof}
    The existence of $L_{ij}$ follows from the fact that $\{\eta_{ij}\otimes \eta_{ij}\}$ is a basis for $\Sym_n$, the existence of $r_D$ from the linearity of the maps $L_{ij}$ and $L_{ij}(H_0) =1$.
\end{proof}
As mentioned in Subsection~\ref{ss:introibp}, one of the error terms arising in the iteration is not directly amenable to integration by parts, as it consists of a symmetric matrix multiplied by a highly oscillatory periodic function with non-zero mean. Following \cite{CaoHirschInauen25}, we overcome this difficulty by subtracting the average of the oscillatory factor. This produces an additional term of the form 
$ 2\pi/\nu^2\,\nabla a \otimes \nabla a,$
which no longer carries rapid oscillations. 

The resulting term can be treated by reabsorbing it into the primitive decomposition. Since gradients of the coefficients appear, this absorption is not purely algebraic. Nevertheless, because no highly oscillatory factor is present, a Picard-type iteration allows one to incorporate the term into the decomposition up to an arbitrarily small error. We refer to \cite[Lemma 2.2]{CaoHirschInauen25} for the detailed argument.

\begin{lemma}\label{l:babykaellen} Let $N_K\in \N$. There exists $r_K=r_K(n, N_K)>0$ such that the following holds for any $1\leq \lambda_0\leq \lambda_1\leq \cdots \leq \lambda_n$ and $H\in C^\infty (\bar \Omega,\Sym_n)$ satisfying
\begin{align}
    &\|H-H_0\|_0 + \frac{\lambda_0}{\lambda_1}< r_K\,,  \label{a:H}\\ 
    &[H]_k \leq \lambda_0^k \text{ for } k=1,\ldots, N_K+1.\label{a:Hestimates}
\end{align} 
For any $J=0,\ldots, N_K$, there exists a vector $a^J=(a^{J}_{11},\ldots, a^J_{nn}) \in C^\infty(\bar \Omega, \R^{n_*})$ and an error term $E^J\in C^\infty (\bar \Omega, \Sym)$ such that 
    \begin{equation}\label{e:Kaellendecomp}
        H= \sum_{1\leq i\leq j\leq n}^{n_*} (a^J_{ij})^2 \eta_{ij}\otimes \eta_{ij} + \sum_{l=1}^n \frac{2\pi}{\lambda_l^2}\nabla a^J_{1l}\otimes \nabla a^J_{1l} + E^J\,,
    \end{equation}
    and the estimates 
     \begin{align}
         &a^J_{ij}\geq r_K, \text{ for }1\leq i\leq j\leq n\,\label{e:ajlower}\\
         &\|a^J\|_k \leq C \lambda_0^k \text{ for }k=0,\ldots, N_K-J+1\,,\label{e:ajestimate}\\
         & \|E^J\|_k \leq C \left(\frac{\lambda_0}{\lambda_1}\right)^{2(J+1)}\lambda_0^k \text{ for } k=0,\ldots,N_K-J \,,\label{e:ejestimate}
     \end{align}
     hold. The constant $C$ only depends on $n, N_K$. 
\end{lemma}

\subsection{Integration by parts}\label{ss:ibp}
In the integration by parts argument, the following algebraic decomposition of a matrix is crucial.
\begin{lemma}\label{l:algebraicdecomp}
Let $i\in\{1,\ldots,n\}$ and let $\xi\in\mathbb S^{n-1}$ satisfy
\[
\xi_k=0 \quad \text{for all } k<i,
\qquad 
\xi_i\neq 0,
\]
(where the first condition is void if $i=1$).
Then there exist linear maps
\begin{align*}
\alpha &: \Sym_n \to \mathbb R^n,\qquad \alpha(M)_k=0 \text{ for } k<i,\\
\pi_m  &: \Sym_n \to \{M\in\Sym_n : M_{kl}=0 \text{ for all } k,l\ge i\},\\
\pi_l &: \Sym_n \to \mathcal V_{i+1},
\end{align*}
such that for every $M\in \Sym_n$ one has the unique decomposition
\begin{equation}\label{e:algebraicdecomposition}
M = \alpha(M)\odot \xi + \pi_m(M) + \pi_l(M).
\end{equation}
\end{lemma}

\begin{proof}
    Observe that the sets $\{M\in\Sym_n : M_{kl}=0 \text{ for all } k,l\ge i\}$ and $\mathcal V_{i+1}$ are vector subspaces of $\Sym_n$ of dimensions $ n_1 = n_* - (n-i+1)_*$ and $n_2 = (n-i)_*$ respectively. Consider then the linear map $\Phi: \R^{n-i+1}\times \R^{n_1}\times \R^{n_2} \to \Sym_n$ defined as
     \[\Phi(\tilde \alpha, \beta,\gamma) = \begin{pmatrix}
         0\\
         \tilde \alpha
     \end{pmatrix}\odot \xi + \sum_{k<i, l\geq k} \beta_{kl}e_k\odot e_l +\sum_{i+1\leq k\leq l\leq n}\gamma_{kl}\eta_{kl}\otimes\eta_{kl}. \]
     It holds
     \[n-i+1 +n_1+n_2 =n -i+1 + n_* - (n-i+1)_*+(n-i)_* = n_*\,,\]
     since 
     \[ (n-i+1)_* = n-i+1 +(n-i)_*.\]
     Thus $\Phi$ is a linear map between vector spaces of the same dimension. In order to show the claim, it therefore suffices to show that $\Phi$ is injective. 
     Let therefore $\Phi(\tilde \alpha, \beta,\gamma) =0$. Observe that the first and the third matrix is supported in indices $i\leq k\leq l\leq n$ and $i+1\leq k\leq l \leq n $ respectively. Hence, for all indices $k<i, l\geq k$,
     \[ 0 = e_k^T \Phi(\tilde \alpha, \beta,\gamma)e_l = b_{kl}\,. \]
     On the other hand, 
     \[0=e_i^T\Phi(\tilde \alpha,\beta,\gamma)e_i =  2\tilde \alpha_1\xi_i \]
     from which $\tilde \alpha_1=0$ follows in view of $\xi_i\neq 0$. Now looking at entry $il$ for $l>i$ yields $\tilde\alpha= 0$ in view of $\xi_i\neq 0$ and $\tilde \alpha_1 =0$:
     \[ \tilde \alpha_{l-i+1}\xi_i =\tilde \alpha_{l-i+1}\xi_i + \tilde \alpha_1\xi_l = e_i^T\begin{pmatrix}
         0\\
         \tilde \alpha
     \end{pmatrix}\odot \xi e_l  = 0\,.\]
Thus 
\[ 0 =\Phi(\tilde\alpha,\beta,\gamma) =\sum_{i+1\leq k\leq l\leq n}\gamma_{kl}\eta_{kl}\otimes\eta_{kl}, \]
which implies $\gamma= 0$, since $\{\eta_{kl}\otimes\eta_{kl}\}$ is a linearly independent family. \end{proof}
We now formulate the iterative integration by parts proposition. Its assumptions should be compared with the heuristic discussion in Subsection~\ref{ss:intronovelties}. There, the frequencies $\lambda,\mu,\nu$ correspond to $\nu_0,\nu_{i-1},\nu_i$, respectively, while the unit vectors $\xi,\eta$ correspond to $\eta_i,\eta_{i-1}$. The vector field $A(t,x)$ plays the role of $\gamma''(t)a_i(x)a_{i-1}(x)\eta_{i-1}$.

\begin{proposition}\label{p:IBP}
    Let $i\in \{1,\ldots,n\}$, let $\xi, \eta\in \S^{n-1}$ satisfy 
    \[\xi_k= \eta_k =0 \text{ for all }k<i,\qquad  \xi_i\neq 0\,,\]
   and consider a vector-field $A\in C^\infty(\S^1\times\bar\Omega,\R^n) $ with
    \begin{equation}
          A_k= 0 \text{ for all } k<i.\label{a:A1}
    \end{equation}
Let $\gamma\in C^\infty(\S^1)$ be  periodic with zero mean.  

Fix frequencies $\nu \geq \mu\geq\lambda\geq 1$. Then for any natural $J\in \N$ there exist vectorfields \[ w^J\in C^\infty(\bar \Omega,\R^n), \qquad A^J\in C^\infty(\S^1\times\bar \Omega,\R^n),\] a periodic $\gamma^J\in C^\infty(\S^1)$ with mean zero, and  remainders
\[ \mathcal F^J_{m},\mathcal F^J_{l}\in C^\infty(\bar \Omega,\Sym_n)\] such that 
    \begin{align}\label{e:IBPdecomp}
        \gamma(\nu x\cdot \xi) A(\mu x\cdot \eta, x) \odot \eta = 2 \,\sym( Dw^J) + \Big(\tfrac{\mu}{\nu}\Big)^J \gamma^{J}(\nu x\cdot \xi) A^J(\mu x\cdot \eta,x) \odot \eta + \mathcal F^J_{m} +\mathcal F^J_{l}, 
    \end{align}
   and \begin{enumerate}[(i)]
        \item $A^J $ satisfies \eqref{a:A1},
        \item $ \mathcal F^J_{m}$ is the mixed-block part of the remainder: 
        \begin{equation}\label{e:mb}
            \mathcal F^J_{m}\in \{M\in \Sym_n : M_{kl} = 0 \text{ for all } k,l\geq i\},        \end{equation}
            \item   $\mathcal F^J_{l}$ is the lower-block part of the remainder: 
            \begin{equation}\label{e:lb}
             \mathcal F^J_{l}   \in \mathcal V_{i+1}.
            \end{equation}
    \end{enumerate} 
    Moreover, if $N_I\in \N$ and   $A$ satisfies 
    \begin{equation}
        \|\partial_t^lD_x^m A\|_0 \leq C_A \lambda^m \text{ for all } l+m\leq N_I, \label{a:A2}
    \end{equation}
    for some constant $C_A\geq 1$,
  then for any $J\leq N_I$ the following estimates hold 
     \begin{align}
     &\|\partial_t^lD_x^m A^J\|_0 \leq  C\lambda^m,\qquad \text{ for all } 0\leq l+m\leq N_I-J, \label{a:A2j}\\
         &\|\mathcal F^J_{m} \|_0\leq C\frac{\lambda}{\nu }\,,\label{e:F_mestimate}\\
        & \left[\mathcal F^J_{l} \right]_k \leq C\nu^k\,,\qquad \text{ for all }k=0,\ldots,N_I-J\label{e:F_lestimates}\\
        & \left[ w^J\right]_k \leq C\nu^{k-1}\,,\quad \text{ for all }k=0,\ldots,N_I-J+1,\label{e:westimates}
     \end{align}
     where the constant $C $ only depends on $n,N_I, \gamma, \xi_i$ and $C_A$.
\end{proposition}

\begin{proof} We prove the proposition by induction on $J$. For $J= 1 $ we start by decomposing
    \begin{equation*}
        \eta= \frac{\eta_i}{\xi_i}\xi + \eta^\bot\,,\qquad
        A = \frac{A_i}{\xi_i}\xi + A^\bot,
    \end{equation*}
    so that $\eta^\bot$ and $A^\bot$ satisfy 
    \begin{equation}\label{e:later}
    \eta^\bot_k = A^\bot_k = 0 \text{ for all } k\leq i.
    \end{equation}
  We can therefore write, omitting the arguments $\nu x\cdot \xi $ and $(\mu x\cdot \eta, x)$ of $\gamma $ and $A$, 
  \begin{align*} \gamma A \odot \eta &= \gamma \frac{\eta_i}{\xi_i} A \odot \xi + \gamma A \odot \eta^\bot\\
  & = \gamma \frac{\eta_i}{\xi_i} A \odot \xi + \gamma\frac{A_i}{\xi_i}\xi\odot  \eta^\bot + \gamma A^\bot \odot \eta^\bot\\ 
  &= \gamma\left(\frac{\eta_i}{\xi_i} A +\frac{A_i}{\xi_i}\eta^\bot\right)\odot \xi +\gamma A^\bot \odot \eta^\bot\,.
   \end{align*}
   Observe that by \eqref{e:later} it holds $\gamma A^\bot \odot \eta^\bot\in \mathcal V_{i+1}$, so that this term will contribute to $\mathcal F_l^1$. For the first term, we abbreviate \begin{equation}\label{e:tilde}
       \tilde A(t,x) =\frac{\eta_i}{\xi_i} A(t,x) +\frac{A_i(t,x)}{\xi_i}\eta^\bot.
   \end{equation} Note that by assumption \eqref{a:A1}, it holds $\tilde A_k = 0 $ for all $k<i$. We can now write
   \begin{align*}
       \gamma(\nu x\cdot \xi ) \tilde A(\mu x\cdot \eta ,x ) \odot \xi  &= 2\,\sym\left( D\left( \frac{\gamma^1(\nu x\cdot \xi) }{\nu } \tilde A(\mu x\cdot \eta ,x)\right)\right) \\
       &\qquad- 2\frac{\gamma^1(\nu x\cdot \xi) }{\nu } \sym\left( \mu \partial_t \tilde A \otimes \eta + D_x \tilde A\right)\,, 
   \end{align*}
   where $\gamma^1\in C^\infty(\S^1) $ is the primitive of $\gamma$ with zero mean
   \[\gamma^1(t) = \int_0^t\gamma(s)\,ds - \frac{1}{2\pi}\int_0^{2\pi} \left( \int_0^t\gamma(s)\,ds\right)dt\,\]
    and $\partial_t \tilde A$ and $D_x\tilde A $ are evaluated at $(\mu x\cdot \eta,x)$. By Lemma \ref{l:algebraicdecomp} we can now write 
    \[\sym(D_x \tilde A) = \alpha\left(\sym(D_x \tilde A)\right)\odot \eta + \pi_m(\sym(D_x \tilde A))+\pi_l(\sym(D_x \tilde A)),\]
    so that when combining with the above, we find \eqref{e:IBPdecomp} for $J=1$ when defining 
    \begin{align}& w^1(x) = \frac{\gamma^1(\nu x\cdot \xi )}{\nu}\tilde A(\mu x\cdot \eta,x),\quad A^1(t,x) =- \partial_t \tilde A(t,x) -\frac{2}{\mu}\alpha\left(\sym(D_x\tilde A(t,x))\right),\nonumber \\
    & \mathcal F^1_m =-2\frac{\gamma^1(\nu x\cdot\xi)}{\nu}\pi_m\left(\sym(D_x\tilde A)\right),\quad \mathcal F^1_l = \gamma A^\bot \odot \eta^\bot -2\frac{\gamma^1(\nu x\cdot\xi)}{\nu}\pi_l\left(\sym(D_x\tilde A)\right)\,,\label{e:IBPfirstchoice}\end{align}
where $A^\bot$ and $D_x\tilde A$ are evaluated in $(\mu x\cdot \eta, \eta)$.
    By definition of $\tilde A$ and the property of the linear map $\alpha$, $A^1 $ satisfies \eqref{a:A1}. Similarly, \eqref{e:mb} and \eqref{e:lb} are satisfied for $J=1$. 

Now let $w^J, A^J, \gamma^J, \mathcal F_m^J,\mathcal F_l^J$ satisfy \eqref{a:A1}--\eqref{e:lb} for some $J\in \N$.
Let $\gamma^{J+1}$ be the primitive of $\gamma^J$ with zero mean and define $w, A^{J+1}, F_m, F_l$ by \eqref{e:IBPfirstchoice} with $\gamma^1$ and $A$ replaced by $\gamma^{J+1}$ and $A^J$ respectively (recall the $\tilde\,$-abbreviation in \eqref{e:tilde}).  We then have 
\[    \gamma^J(\nu x\cdot \xi) A^J(\mu x\cdot \eta, x) \odot \eta = 2 \,\sym( Dw) + \frac{\mu}{\nu} \gamma^{J+1}(\nu x\cdot \xi) A^{J+1}(\mu x\cdot \eta,x) \odot \eta + F_m+F_l,\]
so that \eqref{a:A1}--\eqref{e:lb} follow for $J+1$ upon choosing
 \[w^{J+1} = w^J+ \left( \tfrac{\mu}{\nu}\right)^Jw\,,\quad \mathcal F_m^{J+1}= \mathcal F_m^J+\left( \tfrac{\mu}{\nu}\right)^JF_m\,,\quad \mathcal F_l^{J+1}= \mathcal F_l^J+\left( \tfrac{\mu}{\nu}\right)^JF_l\,.  \]

Assume now that $A$ satisfies the bounds \eqref{a:A2}. From the iterative definition of $A^J$ we have for any $1\leq j\leq J$ 
\[A^j(t,x)=-\partial_t\tilde A^{j-1}(t,x) - \frac{2}{\mu}\alpha\left(\sym\left(D_x \tilde A^{j-1}(t,x)\right)\right)\,,\]
with $A^0=A$, so that 
\begin{align*}
    \|\partial_t^l D_x^m A^j\|_0&\leq C\left(\|\partial_t^{l+1}D_x^m \tilde A^{j-1}\|_0 +\mu^{-1}\|\partial_t^lD_x^{m+1}\tilde A^{j-1}\|_0\right)\\
    &\leq C\left(\|\partial_t^{l+1}D_x^m  A^{j-1}\|_0 +\mu^{-1}\|\partial_t^lD_x^{m+1} A^{j-1}\|_0\right)\\
    &\leq C\lambda^m
\end{align*}   for all $l+m\leq N_I-j$, assuming $A^{j-1}$ satisfies \eqref{a:A2j}. Since $A^0 = A $ satisfies \eqref{a:A2}, \eqref{a:A2j} follows. 

    For the estimates \eqref{e:F_mestimate} and \eqref{e:F_lestimates}, recall that for any $1\leq j\leq J$
    \begin{align*}
        &\mathcal F^j_m(x) = \mathcal F^{j-1}_m(x) -2 \left( \frac{\mu}{\nu}\right)^{j-1}\frac{\gamma^j(\nu x\cdot \xi)}{\nu}\pi_m\left(\sym\left(D_x \tilde A^{j-1}\right)\right)\,,\\
       & \mathcal F^j_l(x) = \mathcal F_l^{j-1}(x) + \left(\frac{\mu}{\nu}\right)^{j-1}\left( \gamma^{j-1}(\nu x\cdot \xi) (A^{j-1})^\bot \odot \eta^\bot -2\frac{\gamma^j(\nu x\cdot \xi)}{\nu}\pi_l\left(\sym\left(D_x\tilde A^{j-1}\right)\right)\right)\,.
    \end{align*}   
with $\mathcal F^0_m=\mathcal F^0_l=0$, and where $(A^{j-1})^\bot $ and $D_x \tilde A^{j-1}$ are evaluated at $(\mu x\cdot\eta,x)$.
 Thus, estimate \eqref{e:F_mestimate} follows from \eqref{a:A2j} and the the linearity of $\pi_m .$ 
For the estimates \eqref{e:F_lestimates}, 
we use the Leibnizrule \eqref{e:Leibniz} to find 
\begin{align}\label{e:F_lest}
    \| D^k\mathcal F_l^j\|_0&\leq \| D^k\mathcal F_l^{j-1}\|_0+C\left( \nu^k \|A^{j-1}\|_0  + \|D^k( A^{j-1}(\mu x\cdot \eta,x))\|_0 \right) \nonumber\\ &\qquad+ C\nu^{-1} \left( \nu^k \|D_x A^{j-1}\|_0 + \|D^k\left( D_x A^{j-1} (\mu x\cdot \eta, x)\right)\|_0\right)\,.
\end{align} 
By  the chain rule \eqref{l:composition} it holds for all $k\leq N_I-(j-1)$
\[ \|D^k\left(A^{j-1}(\mu x\cdot \eta, x)\right)\|_0 \leq C \sum_{l=0}^k \mu^l\|\partial_t^lD_x^{k-l} A^{j-1}(t,x)\|_0\leq C \sum_{l=0}^k \mu^l \lambda^{k-l}\leq C\mu^k, \]
using $\lambda\leq \mu$, and analogously for all $k\leq N_I-j$
\[ \|D^k\left(D_x A^{j-1}(\mu x\cdot \eta, x)\right)\|_0 \leq C \sum_{l=0}^k \mu^l\|\partial_t^lD_x^{k-l+1} A^{j-1}(t,x)\|_0\leq C\sum_{l=0}^k \mu^l \lambda^{k-l+1}\leq C\lambda \mu^k. \]
Plugging into \eqref{e:F_lest} yields \eqref{e:F_lestimates} in view of $\mu\leq \nu$ and $\mathcal F_l^0 = 0$.

The same estimate also yields \eqref{e:westimates}: since for all $1\leq j\leq J$
\[  w^j(x) =w^{j-1}(x) + \left(\frac{\mu}{\nu}\right)^{j-1} \left( \frac{\gamma^j(\nu x\cdot\xi)}{\nu}\tilde A^{j-1}(\mu x\cdot\eta,x)\right)\,,\]
with $w^0 =0$, we find 
\begin{align*} \left[ w^j\right]_k &\leq \left[ w^{j-1}\right]_k+ C\left( \nu^{k-1} + \nu^{-1} \|D^k\left( A^{j-1}(\mu x\cdot \eta, x)\right)\|\right) \\
&\leq \left[ w^{j-1}\right]_k+ C (\nu^{k-1}+ \nu^{-1}\mu^k)\\
&\leq\left[ w^{j-1}\right]_k+ C\nu^{k-1}\,,\end{align*}
for all $k\leq N_I-(j-1)$, finishing the proof.
\end{proof}

\section{Step perturbation}\label{s:step}
In this section we define our perturbation (Proposition \ref{p:step}). It is constructed from normal and tangential vector fields together with suitable oscillatory functions, whose definitions and properties we collect in the following preliminary subsection.

\subsection{Preliminaries}
Let $\rho\geq1$. We will say that a smooth map $u\in C^\infty(\bar \Omega,\R^n) $ satisfies $(P_\rho)$ if 
\begin{equation}\label{a:Prho}
    \frac{1}{\rho}\mathrm{Id}\leq Du^TDu\leq \rho  \mathrm{Id},
\end{equation}
i.e., the matrices $\rho\mathrm{Id}-Du^T Du$ and $Du^TDu-\rho^{-1}\mathrm{Id} $ are positive semi-definite. 
Note that due to the lower bound, such a map is an immersion. Moreover, by taking the trace we see that \eqref{a:Prho} implies 
\begin{equation}\label{e:C1bound}
    \|Du\|_0 \leq \sqrt{n\rho} \,.
\end{equation}
The quantitative immersion condition $(P_\rho)$ guarantees the existence of a well-behaved normal vector field and allows one to define a tangential projection with controlled derivatives. These objects will be used to construct the perturbation.

    \begin{lemma}[Normal vectorfields]\label{l:normals} Let $\Omega\subset \R^n$ be an open, bounded set and let $u\in C^\infty(\bar\Omega,\R^{n+1})$ satisfy $(P_\rho)$ for some $\rho\geq 1$. Then there exists a normal vectorfield $\zeta=\zeta_u\in C^\infty(\bar\Omega,\R^{n+1})$ such that 
    \begin{equation}\label{e:normal}
        |\zeta|= 1\,,\quad Du^T \zeta = 0\,
    \end{equation}
    and such that for any $k\in \N_0$, the estimates
    \begin{equation}\label{e:normalestimates}
            [\zeta]_k\leq C_k[u]_{k+1}
        \end{equation}
hold, with constants $C_k$ depending only on $n,k$ and $\rho$. Moreover, if $v\in C^1(\bar\Omega,\R^{n+1})$ satisfies $(P_{\tilde \rho})$, 
then the normal vector field $\zeta_v$ satisfies \begin{equation}\label{e:normalcloseness}
           \|\zeta_u-\zeta_v\|_0\leq C[v-u]_1
        \end{equation}
        for a constant only depending on $n$ and $\max\{\rho,\tilde\rho\}$.
    \end{lemma}
    \begin{proof}
        We let 
        \[\tilde\zeta = \star(\partial_1 u \wedge \partial_2 u\wedge\ldots\wedge\partial_n u)\,,\]
        where $\star$ denotes the Hodge-star in $\R^{n+1}$. Since $u$ satisfies $(P_\rho)$, it is an immersion and therefore $\tilde \zeta\neq 0$. In fact, we have
        \[ |\tilde \zeta| = \sqrt{ \det\left(Du^TDu\right)} \geq \rho^{-n/2}\,.\] 
      We can therefore set
        \[ \zeta = \frac{\tilde\zeta }{|\tilde \zeta|}\,.\]
       Since by the properties of the Hodge-star 
       \[ \partial_i u \cdot \tilde \zeta \mathrm{vol}_{\R^{n+1}} = \partial_i u\wedge \partial_1 u\wedge\ldots\wedge \partial_n u = 0\]
       for all $i=1,\ldots,n$, \eqref{e:normal} holds.
       To show \eqref{e:normalestimates}, we recall \eqref{e:C1bound} and observe that the map \[F:\left(\R^{n+1}\right)^n\to \R^{n+1}\,,  (v_1,\ldots,v_n) \mapsto \frac{\star(v_1\wedge\ldots\wedge v_n) }{\sqrt{\det\left((\langle v_i,v_j\rangle )_{ij}\right)}}\] is analytic in the set $V_\rho=\overline B_{\sqrt{n\rho}}^{(n+1)n}\cap \{\det\left(\langle v_i,v_j\rangle \right)\geq \rho^{-n/2}\} $, with bounded derivatives of all orders (where the bound only depends on $n, \rho$ and the order of the derivative). Thus it follows from Lemma \ref{l:composition} that 
       \[[\zeta]_k = \left[F\vert_{V_\rho}\circ (\partial_1u,\ldots,\partial_n u)\right]_k \leq C [u]_{k+1}\left( \left[F\vert_{V_\rho}\right]_1 + \|Du\|_0^{k-1}\left[F\vert_{V_\rho}\right]_{k}\right)\leq C_k[u]_{k+1}\]
       for a constant only depending on $n, k$ and $\rho $. 
       
       Lastly, \eqref{e:normalcloseness} follows from the mean-value inequality : 
       \[\| \zeta_u-\zeta_v\|_0  = \|F(Du)-F(Dv)\|_0 \leq \|DF\|_0\|Du-Dv\|_0\leq C[v-u]_1\,,\]
       where the supremum in $\|DF\|_0$ is taken over the convex set $V_{\max\{\rho,\tilde\rho\}}$. This finishes the proof.   
    \end{proof}

     \begin{lemma}[Tangential map]\label{l:tangential}
        Let $u\in C^\infty(\bar\Omega,\R^{n+1})$ satisfy $(P_\rho)$ for some $\rho\geq 1$. Then the tangential map 
        \begin{equation}
            T = Du\left(Du^T Du\right)^{-1} :\Omega\to \R^{(n+1)\times n}
        \end{equation}
        is well-defined and satisfies 
        \begin{equation}\label{e:Testimates}
            [T]_k\leq C_k[u]_{k+1}
        \end{equation}
        for all $k\in \N_0$, with constants $C_k$ only depending on $n,k$ and $\rho$.
    \end{lemma}
    \begin{proof}
        Clearly, the map is well-defined thanks to assumption \eqref{a:Prho}. Moreover, the map $Du\mapsto Du(Du^TDu)^{-1}$ is algebraic in the entries, so that \eqref{e:Testimates} follows from Lemma \ref{l:composition} similarly to the proof of \eqref{e:normalestimates}.
    \end{proof}

    The oscillatory component of the perturbation is built from special periodic functions satisfying the convex integration inclusion \eqref{e:inclusion}. The following lemma provides such functions.

    \begin{lemma}\label{l:oscillatory functions}
        There exist functions $\gamma_1,\gamma_2\in C^{\infty}(\S^1)$ with zero mean such that for all $t\in \S^1$
       \begin{equation}\label{e:inclusion}
    2\gamma_1'(t) +(\gamma_2'(t))^2 = 1\,.
\end{equation}
    \end{lemma}

    \begin{proof}
        It suffices to take 
        \begin{equation}\label{d:corrugations}
\gamma_1(t) = -\frac{1}{4}\sin(2t) \,, \gamma_2(t) = \sqrt{2} \sin(t)\,.\qedhere
\end{equation}
    \end{proof}

\subsection{Step perturbation}

We now introduce the basic perturbation step of the iteration. Its ansatz is the same as in \cite[Lemma 3.1]{CaoHirschInauen25}, however the estimates are finer. Combining the geometric objects from Lemmas~\ref{l:normals} and \ref{l:tangential} with the oscillatory functions from Lemma~\ref{l:oscillatory functions}, the high-frequency perturbation adds a primitive metric in the direction $\eta$ upto some error terms and a symmetric derivative.  The perturbation involves three frequency scales. The highest frequency $\nu$ corresponds to the newly added oscillation. The map $u$ is itself assumed to oscillate at frequency $\mu\le \nu$, while the coefficient $a$ varies at the lower frequency $\lambda\le \mu$. This separation of scales is essential for the integration by parts argument to yield a gain starting from the second family onward (see Proposition~\ref{p:substage2}). In applying the perturbation, the corrector field $w$ is constructed using Proposition~\ref{p:IBP} in order to cancel the terms on the second line of \eqref{e:inducedmetric} up to higher-order errors and terms supported in a lower-dimensional block. In particular, $w$ is expected to oscillate at frequency $\nu$, which motivates assumption~\eqref{a:stepw}.
Lastly, in \eqref{e:sharpD2} we isolate the biggest contribution to the growth of the $C^2$ norm due to the addition of the perturbation (compare with the heuristics in Subsection \ref{ss:intronovelties}).

\begin{proposition}[Step perturbation]\label{p:step}
    Let $\Omega\subset\R^n$ be a open and bounded. Assume that $u\in C^\infty(\overline \Omega,\R^{n+1}) $  satisfies $(P_\rho)$ for some $\rho\geq 1$ and let $\zeta$ and $T$ be as in Lemma  \ref{l:normals}, \ref{l:tangential}. For any $0<\delta<1 $, frequency $\nu\geq 1$, unit vector $\eta\in\S^{n-1}$, $a\in C^\infty(\bar\Omega)$ and corrector vector-field $w\in C^\infty(\bar \Omega,\R^n)$, the perturbed map
     \begin{equation}\label{d:perturbation}
         v(x) = u(x) +\delta T(x)\left(\frac{a^2(x)\gamma_1(\nu x\cdot \eta)}{\nu}\eta + w(x)\right) + \delta^{\sfrac{1}{2}}\frac{a(x)\gamma_2(\nu x\cdot \eta)}{\nu}\zeta(x)\,
     \end{equation}
has induced metric given by 
     \begin{equation}\label{e:inducedmetric}
     \begin{split}
           Dv^TDv&= Du^TDu +\delta a^2 \eta\otimes\eta + \delta\frac{2\pi}{\nu^2 } \nabla a\otimes\nabla a + \delta^{\sfrac{3 }{2}} \mathcal R +2\delta\,\sym(Dw)\\
         &\hspace{1cm }  + 2\delta^{\sfrac{1}{2}}\frac{a\gamma_2}{\nu}\sym\left( Du^T D\zeta\right) +\delta \frac{2\gamma_1 +\gamma_2\gamma_2'}{\nu}\sym\left( \nabla(a^2)\otimes \eta\right) +\delta \frac{\gamma_2^2-2\pi}{\nu^2}\nabla a\otimes \nabla a\,,
     \end{split}
 \end{equation}
     for some $\mathcal R\in C^{\infty}(\bar \Omega, \Sym_n)$. 

    Moreover, if $N\in \N$ and  the following estimates are satisfied 
    \begin{align}
    & \|a\|_0\leq \rho\label{a:stepa0}\\
        &[a]_k\leq \lambda^k     \qquad \text{ for }k=1,\ldots,N+1,\label{a:stepa}\\
        & [u]_{k+1} \leq \delta^{\sfrac{1}{2}}\mu^k\qquad \text{ for } k =1,\ldots, N
        +1\label{a:stepu}\\
        &[w]_k \leq \nu^{k-1}\qquad \text{ for } k=0,\ldots, N-j+1,\label{a:stepw}
     \end{align}
    for some $\lambda$ and $\mu$ satisfying \[\nu\geq\mu\geq\lambda\geq \delta^{-\sfrac{1}{2}}\] and $j\leq N$,  then it holds
     \begin{align}
         &\|\mathcal R\|_0 \leq C\label{e:steprest}\\
         &[v-u]_{k} \leq  C \sdel\nu^{k-1}\qquad\text{ for all } k=0,\ldots, N-j+1\,\label{e:stepv-u}
     \end{align}
     for some constant $C$ only depending on $N,n$ and $\rho$. Moreover, for every $1\leq k\leq l\leq n$, it holds
     \begin{equation}\label{e:sharpD2}
         \|\partial^2_{kl}  v- \left(\partial_{kl}^2 u +\sdel \nu a\gamma_2''(\nu x\cdot \eta)\zeta (\eta\otimes\eta)_{kl}\right)\|_0\leq C\sdel\left( \lambda +\sdel \nu\right)\,. 
     \end{equation}
\end{proposition}
\begin{remark}\label{r:osc}
    The constant $\delta $ will later represent the size of the metric defect and can therefore be assumed to be very small. In particular, the error term $\delta^{\sfrac{3}{2}}\mathcal R$, which at first sight seems too large due to the lack of a prefactor $\nu^{-1}$, is in fact sufficiently small. This allows us to choose oscillatory functions $\gamma_1,\gamma_2$ satisfying the simpler relation \eqref{e:inclusion} instead of the more complicated immersion relation (see for example (2.12) in \cite{CDS}).
\end{remark}
\begin{proof}
    We begin by writing
    \begin{equation}\label{e:derivativedecomp}
        \begin{split}
              Dv &= Du + \delta a^2 \gamma_1' T \eta\otimes\eta + \sdel a\gamma_2'\zeta\otimes\eta \\&  \quad + \delta   TDw+ \delta DT  w \\&
         \quad + \delta\frac{\gamma_1}{\nu}   T\eta\otimes\nabla(a^2) + \delta\frac{\gamma_1}{\nu}a^2 DT\eta \\
         &\quad+ \sdel \frac{\gamma_2}{\nu} \zeta\otimes\nabla a + \sdel \frac{\gamma_2}{\nu}aD\zeta\\
         &=: Du +  A_1 +A_2 \\
         & \quad +  B_1 + B_2\\
         &\quad +  C_1 +C_2\\
         &\quad +  D_1 +D_2\,.
        \end{split}
    \end{equation}
    
    Let us abbreviate $A= A_1+A_2,\ldots,D=D_1+D_2$, so that
\[ Dv^TDv - Du^T Du=  2\sym\left(Du^T (A+B+C+D)\right) + (A+B+C+D)^T(A+B+C+D)\,.\]
We now proceed term by term and single out the terms which will contribute to $\mathcal R$ to find the expression \eqref{e:inducedmetric}. 

By definition of $T$ and $\zeta$ we find 
\begin{align*}2\sym(Du^T(A+B+C+D) )& =2 \delta a^2 \gamma_1' \eta\otimes \eta + 2\delta \sym(Dw) \\
&\quad + 2\delta \frac{\gamma_1}{\nu}\sym(\eta\otimes\nabla(a^2)) + 2\sdel \frac{a\gamma_2}{\nu}\sym(Du^TD\zeta) \\ 
&\quad +2\sym(Du^T(B_2+C_2))\,.
&\end{align*}
On the other hand, using $(T \eta)^T \zeta = (TDw)^T \zeta = 0$ we find 
\[ A_1^T A_2 = B_1^T A_2 = C_1^T A_2 = A_1 ^T D_1 = B_1^T D_1 = C_1^T D_1 = 0\,,\]
and using $|\zeta|=1$, and so in particular $\zeta^T D\zeta =0$, we find moreover 
\[ A_2^T D_2 = D_1^T D_2 = 0\,. \]
We can then write 
\begin{align*}
    (A+\ldots+D)^T(A+\ldots+D) &= A^T A +2\sym(A^T(B+C+D))  \\ & \qquad + (B+C+D)^T (B+C+D)\\
    & = A_1^TA_1 + A_2^T A_2+ 2\sym( A_1^T(B+C+D_2))\\ 
   & \qquad + 2\sym (A_2^T(B_2+C_2))  + 2\sym(A_2^T D_1) \\ &\qquad 
   + (B+C+D)^T (B+C+D)
\end{align*}
and 
\[
    (B+C+D)^T (B+C+D)  = (B+C)^T(B+C) + 2 \sym((B+C)^T D ) + D_1^TD_1 + D_2^TD_2\,.
\]
Now notice that 
\begin{align*}
    A_2^T A_2 &= \delta a^2 (\gamma_2')^2\eta\otimes\eta\,,\qquad  2\sym(A_2^T D_1 ) = \delta\frac{\gamma_2\gamma_2'}{\nu}\sym(\eta\otimes\nabla a)\,,  \\
D_1^TD_1 &= \delta \frac{\gamma_2^2}{\nu^2}\nabla a\otimes \nabla a. 
\end{align*} 
Hence, a regrouping yields 
\begin{align*}
    Dv^TDv - Du^T Du &= \delta a^2 (2\gamma_1' +(\gamma_2')^2)\eta\otimes \eta +2\delta\sym(Dw) \\
    &\quad +2\sdel \frac{a\gamma_2}{\nu}\sym\left( Du^TD\zeta\right) + \delta\frac{2\gamma_1 +\gamma_2\gamma_2'}{\nu}\sym(\eta\otimes\nabla a)+\delta \frac{\gamma_2^2}{\nu^2}\nabla a\otimes \nabla a\\
    &\quad+  2\sym(Du^T(B_2+C_2))+ A_1^TA_1 + 2\sym( A_1^T(B+C+D_2))\\ 
   & \quad + 2\sym (A_2^T(B_2+C_2))+(B+C)^T(B+C)+D_2^TD_2\,.
\end{align*}
By \eqref{e:inclusion}, this shows
 \eqref{e:inducedmetric} with 
\begin{equation}\label{d:R}
    \begin{split}
    \delta^{\sfrac{3}{2}}\mathcal R&:=  2\sym(Du^T(B_2+C_2))+ A_1^TA_1 + 2\sym( A_1^T(B+C+D_2)) \\ 
   & \quad + 2\sym (A_2^T(B_2+C_2))+(B+C)^T(B+C)+D_2^TD_2\,.
\end{split}
\end{equation}
Assume now that \eqref{a:stepa0}--\eqref{a:stepw} are satisfied. In particular, from \eqref{e:normalestimates} and \eqref{e:Testimates} we get
\[ [T]_k+ [\zeta]_k \leq C(1+\sdel\mu^k)\quad \text{ for all } k=0,\ldots,N+1,\]
for a constant only depending on $N,n $ and $\rho$. This immediately gives \eqref{e:stepv-u} for $k=0$ in view of the assumptions and $\delta<1$. We now compute with the help of the Leibniz rule \eqref{e:Leibniz} and Lemma \ref{l:composition} 
\begin{align*}
    [A_1]_k&\leq C\delta\left(\nu^k\|a^2T\eta\|_0 + [a^2]_k\|T\|_0 + \|a^2\|_0[T]_k \right) \leq C\delta\left( \nu^k +\lambda^k +1+\sdel \mu^k\right)\\
    &\leq C\delta\nu^k\qquad  \text{ for all } k=0,\ldots, N+1
\end{align*}
for a constant $C$ only depending on $n,N$ and $\rho$.
Proceeding analogously, we get the estimates
\begin{alignat*}{2}
    [A_2]_k&\leq C\sdel \nu^k\hspace{2cm}&& \text{ for all }k= 0,\ldots, N+1,\\
    [B_1]_k&\leq C\delta\nu^k&& \text{ for all }k= 0,\ldots, N-j,\\
     [B_2]_k&\leq C\delta^{\sfrac{3}{2}}\mu\nu^{k-1}&& \text{ for all }k= 0,\ldots, N-j+1,\\
      [C_1]_k&\leq C\delta\lambda\nu^{k-1}&& \text{ for all }k= 0,\ldots, N,\\
       [C_2]_k&\leq C\delta^{\sfrac{3}{2}}\mu\nu^{k-1}&& \text{ for all }k= 0,\ldots, N,\\
        [D_1]_k&\leq C\sdel\lambda \nu^{k-1}&& \text{ for all }k= 0,\ldots, N,\\
         [D_2]_k&\leq C\delta\mu\nu^{k-1}&& \text{ for all }k= 0,\ldots, N,\,
\end{alignat*}
with a constant $C$ only depending on $n,N$ and $\rho$. This yields \eqref{e:stepv-u} in view of the expression \eqref{e:derivativedecomp}, $\lambda\leq\mu\leq \nu$ and $\delta<1$. The bound \eqref{e:steprest} follows as well given \eqref{d:R}.

Lastly, for any fixed $1\leq k\leq l\leq n$, observe that
\[\partial^2_{kl} v-\partial^2_{kl} u = \partial_k(A_2)_{\cdot l} + \partial_k\left( A_1+B+C+D\right)_{\cdot l}\,,\]
where for a matrix $M$,  $M_{\cdot l}$ denotes its $l$-th column. Observe 
\begin{align*}
    \partial_k (A_2)_{\cdot l } &= \partial_k\left(\sdel a\gamma_2'\zeta \eta_l\right) = \sdel \nu a \gamma_2''\zeta \eta_l\eta_k + \sdel \gamma_2' \partial_k( a\zeta \eta_l)\,,
\end{align*} 
so that 
\[ \|\partial^2_{kl} v-\partial^2_{kl} u -\sdel \nu a \gamma_2''\zeta \eta_l\eta_k\|_0\leq C\left(\sdel[ a\zeta]_1+[A_1+B+C+D]_1 \right)\,. \]
The estimate \eqref{e:sharpD2} then follows from
\[[ a\zeta]_1 \leq C\left(\lambda +\sdel\mu\right)\]
and 
\[[A_1+B+C+D]_1 \leq C\sdel \left( \lambda +\sdel\nu\right)\,.\]
This concludes the proof.
\end{proof}

\section{Iterative Proposition: Stage}\label{s:stage}

In this section we state and prove the main iteration proposition. It allows us to construct, from an immersion $u$ whose induced metric is already sufficiently close to the target metric $g$, a new immersion $v$ whose metric deficit is decreased by the factor $\Lambda^{-1} := K^{-J}$. This improvement comes at the cost of increasing the $C^2$ norm by a factor $K^{J(n-1)+n^2} = \Lambda^{n-1}\Lambda^{n^2/J}$, where the latter factor is negligible for large $J$. This formalizes the heuristic argument outlined at the end of Subsection \ref{ss:intronovelties}. The "loss of domain" from $U$ to $V$ arises from a mollification step. Compare the exponent of $K$ in \eqref{e:C2} to the corresponding one in  \cite{CaoHirschInauen25}[Proposition 3.4]. 

    \begin{proposition}[Stage]\label{p:stage} Let $U\subset \R^n$ be a bounded open set. For any $J\in\mathbb{N},$ there exist positive constants $\delta_*(n, J, g),\, r_*(n, J, g), \,c_*(n, J, g)$ such that the following holds.
    
If $u\in C^2(\bar U,\R^{n+1})$ is an immersion such that 
    \begin{align}
        &\|g- Du^TDu - \delta H_0 \|_0 \leq r\delta\,, \label{a:metricdeficit}\\
        &\|u \|_2 \leq \delta^{\sfrac{1}{2}}\lambda,\label{a:C2}
    \end{align}
for some constants $r>0, \, 0<\delta\leq1$ and $\lambda\geq \delta^{-\sfrac{1}{2}}$, and if 
\[ r\leq r_*\,,\qquad \delta\leq \delta_*,\]
then for any bounded open  $V \Subset U$, and any 
\[\hat\delta<\frac{r\delta}{|H_0|},  \quad K\geq c_*\]
there exists an immersion $v\in C^2(\overline{V}, \R^{n+1})$ such that 
\begin{align}
        &\|g- Dv^tDv - \hat\delta h_0 \|_0 \leq C\delta\left(K^{-J}+\delta^{\sfrac{1}{2}}\right) + \lambda^{-2}\,, \label{e:metricdeficit}\\
        &\|v-u\|_k \leq  C \delta^{\sfrac{1}{2}}\lambda^{k-1}\,, \text{ for }k=0,1\,, \label{e:C1}\\
        &\|v\|_2 \leq C\delta^{\sfrac{1}{2}}\lambda d^{-1}K^{J(n-1)+n^2}\,,\label{e:C2}
\end{align}
 where $d=\min\{1, \mathrm{dist}(V, \partial U)\}$ and the constant $C$ only depends on $n, g, J$.
\end{proposition}

We prove Proposition \ref{p:stage} by following the procedure outlined in the introduction. First, we decompose a re-scaled and mollified version of the metric defect using the decomposition of Lemma \ref{l:babykaellen}. We then iteratively add perturbations in the form \eqref{d:perturbation} to add the corresponding primitive metrics. Within each subfamily $\{\eta_{ij}\otimes \eta_{ij}\}_{j\ge i}$ of primitive metrics, we use the iterative integration by parts mechanism to choose suitable corrector vector fields $w$, producing sufficiently small errors upto remainders that lie in $\mathcal V_{i+1}$ and which can later be canceled exactly.  

Since the process of adding primitive metrics differs slightly for the first subfamily ($i=1$) and the subsequent ones ($i\ge 2$), (due to the additional errors $\mathcal F_m$ arising in the iteration by parts for $i\geq 2$),  it is convenient to divide the argument into two substages. These substages, which respectively handle the addition of all primitive metrics of family 1 and of families $i\ge 2$, are stated and proven in the following subsection. Proposition \ref{p:stage} is then obtained by iterating over these substages.

\subsection{Substages}

In both substages we fix a large natural number $N^0$ resp. $N^{i-1}\in \N$, which  denotes the number of derivatives we control. We will in particular assume 
\begin{equation}\label{d:N}
    N^0,N^{i-1}\geq 2n.
\end{equation}
Moreover, we let $\rho \geq 1 $ and assume that $u\in  C^\infty(\bar\Omega,\R^{n+1})$  satisfies $(P_\rho)$. For the fixed $i\in\{1,\ldots,n\}$ we assume moreover that $a_i,\ldots,a_n\in C^\infty(\bar\Omega,\R^{n+1})$ satisfy
\begin{equation}\label{a:subst}
    0\leq a_j\leq \rho \qquad \text{ for all }j=i,\ldots,n\,.
\end{equation}

\begin{proposition}[Substage 1]\label{p:substage1} Let $N^0\in \N$. 
   There exist $K_*(n,N^0,\rho)\geq 1$, $0<\delta_*(n,N^0,\rho)<1$ such that the following holds. Assume $u, a_j$ satisfy the estimates 
    \begin{align}\label{a:subst1}
        &[u]_{k+1} \leq \sdel \lambda^k,
        &[a_j]_k\leq \lambda^k 
    \end{align}
     for all $k=1,\ldots, N^0+1$ and $j=1,\ldots,n$ for some constants $\lambda \geq 1$ and $0<\delta \leq \delta_*$ with $\lambda\sdel\geq 1$. Then for all $K\geq K_*$ and all natural $J\leq N^0/n-1$, there exists $v \in C^\infty(\bar\Omega,\R^{n+1})$ such that 
    \begin{equation}\label{e:subst1induced}
        Dv^TDv - \left( Du^T Du + \delta \sum_{l=1}^n a_l^2 \eta_{1l}\otimes\eta_{1l} + \delta \sum_{l=1}^n\frac{2\pi}{(\lambda K^l)^2} \nabla a_l\otimes\nabla a_l \right) = \delta\mathcal F_1 +\delta \mathcal R_1\,,
    \end{equation}
    for $\mathcal F_1, \mathcal R_1 \in C^\infty(\bar \Omega,\Sym_n)$ such that 
     \begin{align}
        & \|\mathcal R_1 \|_0\leq C\left(K^{-J}+\sdel\right)\,\label{e:subst1R_1}\\
   & \mathcal F_1\in \mathcal V_2\,,\qquad  \|\mathcal F_1 \|_0 \leq  CK^{-1}\,,\qquad [\mathcal F_1 ]_k \leq C \left(\lambda K^n\right)^k \quad \text{ for  } k=1,\ldots, N^1+1, \label{e:subst1F_1}
   \end{align}
    and
\begin{equation}\label{e:subst1v-u}
        \|v-u\|_0 \leq C\sdel \lambda^{-1},\qquad [v-u]_{k+1}\leq C\sdel\left(\lambda K^n\right)^k \quad \text{ for }k=0,\ldots, N^1+1,
    \end{equation} for 
    \begin{equation}\label{d:N^1}
        N^1 = N^0-n(J+1).
    \end{equation}
    The constant $C$ in \eqref{e:subst1R_1}--\eqref{e:subst1v-u}  only depends on $N^0,n$ and $\rho$. 
\end{proposition}

 The proof proceeds by iteratively adding perturbations to add the primitive metric $a_j^2\eta_{1j}\otimes\eta_{1j}$. These perturbations are added with frequencies $\nu_1= K\lambda$, $\nu_j = K\nu_{j-1}$ for $j=2,\ldots,n$ (compare with  $\nu$ in \eqref{e:subst1_parameters}). By using integration by parts we will reduce the size of the resulting error  to $\delta K^{-J}$, up to the remainder $\mathcal F_1 $, which, as $i=1$, lies entirely in $\mathcal V_2$  and can therefore be canceled exactly later. Note the loss of  information on $n(J+1)$ derivatives of the maps, which is due to the integration by parts (compare with \eqref{e:westimates}).
    
\begin{proof}[Proof of Proposition \ref{p:substage1}]
We claim that if $K_*$ is large enough and $\delta_*$ is small enough, we can iteratively for $j=1,\ldots,n$ choose a vector-field $w_j$ and suitable parameters such that $u_j$ defined by \eqref{d:perturbation} with 
\[ u=u_{j-1},\quad a=a_j,\quad \nu=\nu_j=\lambda K^j,\quad \eta=\eta_j,\quad w=w_j\]
satisfies 
  \begin{equation}\label{e:subst1_iterass}
      Du_j^TDu_j - \left( Du^T Du + \delta \sum_{l=1}^j a_l^2 \eta_{1l}\otimes\eta_{1l} + \delta \sum_{l=1}^j\frac{2\pi}{(\lambda K^l)^2} \nabla a_l\otimes\nabla a_l \right) =\delta \mathcal F_1^j +\delta\mathcal R_1^j\,,
  \end{equation}
 for $\mathcal R_1^j$ and $\mathcal F_1^j$ satisfying \eqref{e:subst1R_1} and \eqref{e:subst1F_1} with $j$ replacing $n$, and moreover 
 \begin{equation}\label{e:susbst1_iterassv-u}
     \|u_j-u\|_0\leq C\sdel\lambda^{-1},\qquad [u_j-u]_{k+1}\leq C\sdel (K^j\lambda)^k\,\quad\text{ for } k=0,\ldots,N^0-j(J+1)+1.
 \end{equation}

 Clearly the map $v:= u_n $ then satisfies \eqref{e:subst1induced} and \eqref{e:subst1v-u}, hence it suffices to prove the claim to conclude the proof.

We prove the claim by  a finite induction. Let therefore $j=1$. Define $u_1$ by \eqref{d:perturbation} with 
\[  
   u=u,\quad a= a_1, \quad \eta =\eta_{11},\quad \nu= \nu_1:= K\lambda,
\] and $w=w_1$ to be chosen. It then follows from \eqref{e:inducedmetric} that 
  \begin{align}\label{e:subst1_iterative}& Du_1^TDu_1 -\left(Du^TDu +\delta a_1^2 \eta_{11}\otimes \eta_{11} + \frac{2\pi}{(\lambda K)^2}\nabla a_1\otimes \nabla a_1 \right)   =  \delta^{\sfrac{3}{2}} \mathcal R +2\delta\,\sym(Dw_1)\nonumber\\
         &\hspace{1cm }  + 2\delta^{\sfrac{1}{2}}\frac{a_1\gamma_2}{\nu }\sym\left( Du^T D\zeta\right) +2\delta \frac{2\gamma_1+\gamma_2\gamma_2'}{\nu }\sym\left( \nabla(a_1^2)\otimes \eta_{11}\right) +\delta \frac{\gamma_2^2-2\pi}{\nu^2}\nabla a_1\otimes \nabla a_1.
         \end{align}
         We now want to define $w_1$ as a sum of three corrector vector-fields, each one designed to cancel one of the error terms on the second line with the help of the Proposition \ref{p:IBP}.
Observe that all of these terms are of the form \[\delta\frac{\lambda}{\nu}\gamma(\nu x\cdot \eta) M\,, \]
for a  $\gamma\in C^\infty(\S^1)$ with mean zero and a symmetric matrix field $M\in C^\infty(\bar \Omega,\Sym_n)$ such that for all $k=0,\ldots,N^0$
\[ \|M\|_k\leq C \lambda^k\,\]
for some constant depending only on $n,N^0$ and $\rho$.
Indeed, 
the assumptions \eqref{a:subst}, \eqref{a:subst1} and the estimates \eqref{e:normalestimates} yield
\begin{align*}
    \left[ \frac{a_1}{\sdel \lambda }\sym\left(Du^TD\zeta\right)\right]_k &\leq C\frac{1}{\sdel\lambda }\left( [a_1]_k[u]_2 + [u]_{k+1}[u]_2+[u]_{k+2}\right) \\&
    \leq  C\frac{1}{\sdel\lambda }\left( \sdel \lambda^{k+1}+ \delta \lambda^{k+1}\right)\\
    &\leq C\lambda^k
\end{align*}
 and similarly 
 \begin{align*}
     \left[ \frac{1}{\lambda}\sym\left( \nabla(a_1^2)\otimes \eta_{11}\right) \right]_k \leq C\lambda^{-1}[a_1]_{k+1}\leq C \lambda^k,
 \end{align*}
 and 
  \begin{align*}
     \left[ \frac{1}{\lambda\nu}\nabla a_1\otimes \nabla a_1 \right]_k \leq C\frac{1}{\lambda\nu}[a_1]_1[a_1]_{k+1}\leq C\frac{1}{\lambda\nu} \lambda^{k+2}\leq C\lambda^k\,
 \end{align*}
 in view of $\nu = K \lambda\geq \lambda$.  
 
For any such term, by Lemma \ref{l:algebraicdecomp} with $\xi =\eta_{11}$ we can write 
 \[M = \alpha(M)\odot \eta_{11} + \pi_l(M)\]
 for a remainder $\pi_l(M)\in \mathcal V_2$. By the linearity of $\alpha$ it holds 
 \[[\alpha(M)]_k\leq C_A\lambda^k\quad \text{ for all } k= 0,\ldots,N^0\]
 for a constant $C_A$ only depending  on $n,N^0$ and $\rho$.
Thus, we can  apply Proposition \ref{p:IBP} with $N_I=N^0$, $\xi =\eta=\eta_{11}$, $\mu=\lambda=\lambda$, $\nu=K\lambda$ and $A(t,x)= \alpha(M(x))$. This yields vector-fields $w^J, A^J$ and a remainder $\mathcal F^J_l$ such that 
\[\gamma \alpha(M)\odot \eta = 2 \sym(Dw^J) + \left(\tfrac{\lambda}{\nu}\right)^J\gamma^JA^J\odot\eta_{11} + \mathcal F^J_l\]
 and \eqref{e:lb} and \eqref{a:A2j}--\eqref{e:westimates} are satisfied. In particular, each term on the second line of \eqref{e:subst1_iterative} can be written as 
 \[\delta\frac{\lambda}{\nu}\gamma(\nu x\cdot \eta) M = 2\delta\frac{\lambda}{\nu}\sym(Dw^J) + \delta \left(\tfrac{\lambda}{\nu}\right)^{J+1}\gamma^J A^J\odot\eta_{11} + \delta\frac{\lambda}{\nu}\left( \mathcal F^J_l + \pi_l(M)\right)\,.   \]
 Now observe that 
 \[\left[\tfrac{\lambda}{\nu}w^J\right]_k\leq CK^{-1}\nu^{k-1}\leq\frac13 \nu^{k-1}\qquad \text{ for } k=0,\ldots, N^0-J +1\]
 if $K\geq K_*(n,N^0)$ is large enough, where we used $\nu= K \lambda$.  Moreover, from \eqref{e:F_lestimates} we find 
 \begin{align*} \|\frac{\lambda}{\nu}\left( \mathcal F^J_l + \pi_l(M)\right)\|_0 &\leq CK^{-1}, \\
 \left[\frac{\lambda}{\nu}\left( \mathcal F^J_l + \pi_l(M)\right)\right]_k&\leq C\nu^k\qquad\text{ for }k=0,\ldots,N^0-J \end{align*}
 Define then $w_1$ to be the sum of the vector-fields $-\tfrac{\lambda}{\nu}w^J$ corresponding to each of the three error terms on the second line of \eqref{e:subst1_iterative},  $\mathcal F_1^1$ to be the sum of the three $\tfrac{\lambda}{\nu}\left(\mathcal F^J_l + \pi_l(M)\right)$ terms,  and $\mathcal R_1^1$ to be the sum of the three $\left(\tfrac{\lambda}{\nu}\right)^{J+1}\gamma^J A^J\odot\eta_{11}$ terms  and $\sdel \mathcal R$ (recall the small remainder $\delta^{\sfrac{3}{2}}\mathcal R$ in \eqref{e:subst1_iterative}). The map $w_1$ satisfies \eqref{a:stepw}, so that by \eqref{e:steprest}, $\mathcal R$ is bounded by a constant only depending on $n,N^0$ and $\rho$, and \eqref{e:stepv-u} holds. This yields \eqref{e:subst1_iterass} for $j=1$ in view of $\nu= K \lambda$ and $\|A^J\|_0\leq C$. The estimates \eqref{e:susbst1_iterassv-u} for $j=1$ follow at once from \eqref{e:stepv-u} and $\nu\geq \lambda$.

 Now assume the claim holds up to some $j-1\leq n-1$. Observe that from \eqref{e:susbst1_iterassv-u} it follows that $u_{j-1}$ satisfies $(P_{2\rho})$ if $\delta_*$ is small enough (depending on $n,N^0$ and $\rho$). 

We then define $u_j$ by \eqref{d:perturbation} with 
 \begin{equation}\label{e:subst1_parameters}
u= u_{j-1}, \quad a=a_{j},\quad \nu =\nu_j:=\lambda K^{j}, \quad \eta = \eta_{1j}
\end{equation}
 and $w=w_{j}$ to be chosen. By \eqref{e:inducedmetric} it holds

   \begin{align}\label{e:subst1_iterative2}
   & Du_{j}^TDu_{j} -\left(Du_{j-1}^TDu_{j-1} +\delta a_{j}^2 \eta_{1j}\otimes \eta_{1j} + \frac{2\pi}{(\lambda K^j)^2}\nabla a_j\otimes \nabla a_j \right)   =  \delta^{\sfrac{3}{2}} \mathcal R +2\delta\,\sym(Dw)\nonumber\\
         &\hspace{1cm }  + 2\delta^{\sfrac{1}{2}}\frac{a_j\gamma_2}{\nu }\sym\left( Du^T D\zeta\right) +2\delta \frac{2\gamma_1+\gamma_2\gamma_2'}{\nu }\sym\left( \nabla(a_j^2)\otimes \eta_{11}\right) +\delta \frac{\gamma_2^2-2\pi}{\nu^2}\nabla a_j\otimes \nabla a_j.
         \end{align}
Notice that by \eqref{e:susbst1_iterassv-u}, 
 \begin{equation}
     [u_{j-1}]_{k+1}\leq C\sdel(\lambda K^{j-1})^k\leq \sdel\mu_j^k\, 
 \end{equation} 
for all $k=1,\ldots, N^0-(j-1)(J+1)+1$, if we set
\[\mu_j = C_\star \lambda K^{j-1}\] 
for some constant $C_\star $ depending only on $n,N^0$ and $\rho$. Assume now that 
\[K \geq C_\star\,.\]
This ensure $\lambda\leq \mu_j\leq \nu_j$. We can then infer that all of the terms on the second line of \eqref{e:subst1_iterative2} are of the form 
\[\delta\frac{\mu_j}{\nu_j}\gamma(\nu_j x\cdot\eta) M\]
for a smooth symmetric matrix field $M$ satisfying 
\[[M]_k\leq C \mu^k_j\,,\]
for all $k=0,\ldots,N^0-(j-1)(J+1)$,
using the same reasoning as above and $\lambda \leq\mu_j\leq \nu_j$. Analogously to the $j=1$ case, we can then decompose these error terms with the help of Lemma 2.3 and apply the iterative integration by parts Proposition \ref{p:IBP} with 
\[\zeta=\eta=\eta_{1j},\quad \lambda=\mu = \mu_j,\quad\nu = \nu_j,\quad N_I=N^0-(j-1)(J+1),\quad A(t,x) =\alpha(M(x)).\]This yields the decomposition
 \[\delta\frac{\mu_j}{\nu_j}\gamma(\nu_j x\cdot \eta) M = 2\delta\frac{\mu_j}{\nu_j}\sym(Dw^J) + \delta\left(\tfrac{\mu_j}{\nu_j}\right)^{J+1}\gamma^J A^J\odot\eta_{11} + \delta\frac{\mu_j}{\nu_j}\left( \mathcal F^J_l + \pi_l(M)\right)\,   \]
 for each of these terms. Gathering the corresponding vector-fields $-\tfrac{\mu_j}{\nu_j}w^J$ defines $w_j$, and the estimate 
 \[ [w_j]_k\leq \nu_j^{k-1}\quad\text{ for } k=0,\ldots, N_I-J+1 = N^0-j(J+1)+2\]
 follows as above. In particular, the term $\mathcal R$ in \eqref{e:subst1_iterative2} is bounded by a constant only depending on $n,N^0$ and $\rho$, as is $A^J $ by \eqref{a:A2j}. Moreover 
 \[\|\left(\tfrac{\mu_j}{\nu_j}\right)^{J+1}\gamma^J A^J\odot\eta_{11}\|_0\leq C C_\star^{J+1} K^{-(J+1)}\leq C K^{-J}, \] 
 for a constant $C$ only depending on $n,N^0$ and $\rho$.
 Hence, adding these three terms to $\delta^{\sfrac{1}{2}}\mathcal R$ and $\mathcal  R_1^{j-1}$ constitutes $\mathcal R_1^{j}$. On the other hand, adding the three $\frac{\mu_j}{\nu_j}\left( \mathcal F^J_l + \pi_l(M)\right)$-terms to $\mathcal F_l^{j-1}$ defines $\mathcal F_1^j$. The estimates \eqref{e:subst1F_1} follow from  \eqref{e:F_lestimates} and the induction assumption,  and \eqref{e:susbst1_iterassv-u} follows from the induction assumption and \eqref{e:stepv-u}. This concludes the proof.  
\end{proof}

For the second substage, corresponding to families $i\geq 2$, the iteration by parts procedure behaves slightly differently. Unlike the first family, the mixed-block error $\mathcal F_m$
  produced in the integration by parts step is not automatically contained in $\mathcal V_{i+1}$. Consequently, these errors cannot be canceled immediately and must be made sufficiently small by increasing the initial frequency of the perturbation, reflected in the choice $\nu_i\geq K^J\lambda$. After this initial frequency boost, the remaining perturbations in the subfamily can be added with the standard geometric progression $\nu_j = K\nu_{j-1}$ for $j\geq i$, following the heuristic explanation in Subsection \ref{ss:intronovelties}. Recall that $u$ is assumed to satisfy $(P_\rho)$, and $a_j$ \eqref{a:subst}.
 
\begin{proposition}[Substage 2]\label{p:substage2} Let $i\in\{2,\ldots,n\}$ and let $N^{i-1}\in \N$.
   There exists $K_*(n,N^{i-1},\rho)\geq 1$ and $0<\delta_*(n,N^0,\rho)<1$ such that the following holds. Assume $u, a_j$ satisfy the estimates 
    \begin{equation}\label{a:subst2}
        [u]_{k+1} \leq \sdel \lambda^k,
        \qquad[a_j]_k\leq \lambda^k 
    \end{equation}
     for all $k=1,\ldots, N^{i-1}+1$ and $j=i,\ldots,n$ for some constants $\lambda \geq 1$ and $0<\delta \leq \delta_*$ with $\lambda\sdel\geq 1$. Then for all $K\geq K_*$ and all natural $n\leq J\leq N^{i-1}-n$, there exists $v \in C^\infty(\bar\Omega,\R^{n+1})$ such that 
    \begin{equation}\label{e:subst2induced}
        Dv^TDv - \left( Du^T Du + \delta \sum_{j=i}^n a_j^2 \eta_{ij}\otimes\eta_{ij} \right) = \delta \mathcal F_i +\delta \mathcal R_i\,,
    \end{equation}
    for $\mathcal F_i, \mathcal R_i \in C^\infty(\bar \Omega,\Sym_n)$ such that 
    \begin{align}
        &\|\mathcal R_i \|_0\leq C\left(K^{-J}+\sdel\right)\,\label{e:subst2R_i}\\
        &\mathcal F_i\in \mathcal V_{i+1}\,,\quad  \|\mathcal F_i \|_0 \leq CK^{-1}\,,\quad [\mathcal F_i ]_k \leq C \left(\lambda K^{J+n-i}\right)^k \quad \text{ for } k=1,\ldots, N^{i-1}-(1-\delta_{in})J,\label{e:subst2F_i}
   \end{align}
    and 
\begin{equation}\label{e:subst2v-u}
       \|v-u\|_0\leq C\sdel \lambda^{-1},\quad
       [v-u]_{k+1}\leq C\sdel\left(\lambda K^{J+n-i}\right)^k \quad \text{ for }k=0,\ldots, N^i+1
    \end{equation}
    for 
    \begin{equation}\label{d:N^i}
        N^i = \begin{cases}
            N^{i-1}-1  \quad \text{ if } i=n\\
            N^{i-1 }- J-(n-i) \quad \text{ for } 2\leq i \leq n-1
        \end{cases}
    \end{equation}
     The constant $C$ in \eqref{e:subst2F_i}--\eqref{e:subst2v-u}  only depends on $N,n$ and $\rho$. 
\end{proposition}

\begin{proof}[Proof of Proposition \ref{p:substage2}]
Set \begin{equation}\label{d:subst2_frequencies}
    \nu_i=K^J\lambda, \quad \nu_j=K\nu_{j-1}= K^{J+j-i}\lambda\,\text{ for }j=i+1,\ldots,n.
\end{equation}
We claim that if $K_\star$ is large enough and $\delta_\star$ small enough,  we can iteratively for $j=i,\ldots,n$ choose a suitable vector-field $w_j$ such that when we define $u_j$ by \eqref{d:perturbation} with 
\begin{equation}\label{e:ss2iterchoices}
    u = u_{j-1},\quad a=a_j,\quad \eta= \eta_{ij},\quad \nu= \nu_j,\quad w=w_j
\end{equation}
   it holds 
    \begin{equation}\label{e:subst2_iterass}
        Du_j^TDu_j- \left( Du^TDu +\delta\sum_{k=i}^j a_k^2\eta_{ik}\otimes\eta_{ik}\right) = \delta \mathcal F_i^j +\delta \mathcal R_i^j,
    \end{equation}
    for $\mathcal F_i^j$ and $\mathcal R_i^j$ satisfying \eqref{e:subst2F_i} and \eqref{e:subst2R_i}, and moreover 
     \begin{equation}\label{e:susbst2_iterassv-u}
     \begin{split}
          \|u_j-u\|_0&\leq C\sdel\lambda^{-1},\\
          [u_j-u]_{k+1}&\leq C\sdel \nu_j^k\,\quad\text{ for } k=0,\ldots,N^{i}_j+1,
     \end{split}
 \end{equation}
 for \begin{equation}
     N^{i}_j = \begin{cases}
         N^{i-1}-1\quad \text{ if }j=i\\
         N^{i-1}- J-(j-i) \quad \text{ if }i+1\leq j\leq n
     \end{cases}.
 \end{equation}
 Clearly, if the claim is true, the map $v=: u_n$ and the fields $\mathcal F_i = \mathcal F_i^n, \mathcal R_i = \mathcal R_i^n$ will satisfy the desired properties. 

 We again proceed by finite induction and let $j=i$. Define $u_i$ by \eqref{d:perturbation} with $w=0$, $a=a_i$, $\eta=\eta_{ii}$ and $\nu=\nu_i=K^J\lambda$. Using the assumption \eqref{a:subst2}, estimates \eqref{e:normalestimates} and \eqref{e:steprest} we find that 
 \begin{align*}
     \|Du_i^TDu_i -\left(Du^TDu + \delta a_i^2\eta_{ii}\otimes\eta_{ii}\right) \|&\leq C\delta \left( \lambda^2\nu^{-2} + \sdel + \delta^{-\sfrac{1}{2}}\nu^{-1}[u]_2+\lambda\nu^{-1}\right)\\
     &\leq C\delta\left(\lambda\nu^{-1}+\sdel\right)\\
     & \leq C\delta( K^{-J} + \sdel).
 \end{align*} 
 This shows \eqref{e:subst2_iterass} after choosing $\mathcal F_i^i =0$. Moreover, since $w=0$ we have 
 \[[u_i-u]_{k+1}\leq C\sdel \nu_1^k\quad \text{ for } k=0,\ldots,N^{i-1}\,,\] by \eqref{e:stepv-u}, which yields \eqref{e:susbst2_iterassv-u} for $j=i$.

If $i=n$, the proof is finished. Hence, assume $i\leq n-1$ and that up to some $j-1$ with $ i\leq j-1\leq n-1$, $u_{j-1}$ has been constructed as above. We now want to choose a suitable vector-field $w_j$ such that if we  define $u_j$ by \eqref{d:perturbation} with 
the choices \eqref{e:ss2iterchoices} it holds \eqref{e:susbst2_iterassv-u}. 

Firstly, it follows from the inductive assumption and \eqref{e:subst2_iterass} that when $\delta_* $ is small enough  (depending on $n,N^{i-1}$ and $\rho$), $u_{j-1}$ satisfies $(P_{2\rho})$, hence the normal vectorfield $\zeta$, and the map $T$ in \eqref{d:perturbation} are well-defined and satisfy the estimates 
\begin{equation}\label{e:subst2_zeta}[T]_k+ [\zeta]_k \leq C[u_{j-1}]_{k+1}\leq C(1+\sdel\nu_{j-1}^k)\quad \text{ for }k=0,\ldots,N^i_{j-1}+1\,,\end{equation}
for constant $C$ only depending on $n,N^{i-1},\rho$. Define therefore $u_j$ by \eqref{d:perturbation} with $w=w_j$ to be chosen. We then have
\begin{lemma}\label{l:inductive} For any $k=i,\ldots,j-1 $, define $A_k\in C^\infty(\S^1\times \bar \Omega,\R^n)$ by 
 \begin{equation}\label{d:A_k}
 A_k(t,x) =  -2a_j(x) a_k(x) \gamma_2''(t)\eta_{ik}\,.
 \end{equation}
     Then $A_k$ satisfies \eqref{a:A1} and \eqref{a:A2} with $N_I = N^{i-1}$. Moreover, it holds 
     \begin{equation}\label{e:subst2_inductiveinduced}
         \begin{split}
               Du_j^TDu_j &=  Du_{j-1}^TDu_{j-1} +\delta a_j^2\eta_{ij}\otimes\eta_{ij} + \delta R\\
          & \qquad + 2\delta \sym(Dw)+ \delta\sum_{k=i}^{j-1} \frac{\nu_k}{\nu_j} \gamma_2(\nu_j x\cdot \eta_{ij}) A_k(\nu_k x\cdot \eta_{ik},x)\odot \eta_{ik} 
         \end{split}
     \end{equation}
               for some $R\in C^\infty(\bar \Omega,\Sym_n)$. If moreover
     \begin{equation}\label{e:iterativewbounds}
         [w_j]_l\leq \nu_j^{l-1} \quad\text{ for } l=0,\ldots,N^{i-1}-J+1,
     \end{equation} then  
     \begin{equation}\label{e:inductiveR}
         \|R\|_0\leq C\left(K^{-J}+\sdel\right)\,.
     \end{equation}
 \end{lemma}
We postpone the proof of this lemma and show how to conclude with it. As is apparent, assuming that the bounds \eqref{e:iterativewbounds} will be satisfied by our choice of $w=w_j$, the only error terms which are not small enough with the choice of $\nu_j = K\nu_{j-1}$ are of the form 
\[\delta \frac{\nu_k}{\nu_j}\gamma_2(\nu_j x\cdot \eta_{ij}) A_k(\nu_k x\cdot \eta_{ik},x)\odot \eta_{ik}.\]
This is exactly the setting of Proposition \ref{p:IBP} with $\lambda =\lambda$ (by assumption \eqref{a:subst2} and the definition of $A_k$), $\mu =\nu_k$ and $\nu=\nu_j$, $N_I =N^{i-1}$,  $\gamma=\gamma_2$. Hence, for each $k=i,\ldots,j-1$ we can find smooth vector fields $w_k^J, A_k^J$, periodic $\gamma_k^J$ and remainders $\mathcal F^J_{m,k}, \mathcal F^J_{l,k}$ such that 
\[\gamma_2(\nu_j x\cdot \eta_{ij}) A_k(\nu_k x\cdot \eta_{ik},x)\odot \eta_{ik} = 2 \sym(Dw^J_k) + \left(\tfrac{\nu_k}{\nu_j}\right)^J\gamma^J_kA^J_k\odot\eta_{ik}+\mathcal F^J_{m,k}+ \mathcal F^J_{l,k}\,,\]
and where 
\begin{equation}\label{e:inductivemixedblock}
    \|\mathcal F_{m,k}^J\|_0 \leq C\frac{\lambda}{\nu_j} \leq CK^{-(J+j-i)}\leq CK^{-J}
\end{equation}
and $\mathcal F_{l,k}^J\in \mathcal V_{i+1}$. Let therefore  
\[w_j:= -\sum_{k=i}^{j-1} \frac{\nu_k}{\nu_j}w_k^J\,\]
in the definition \eqref{d:perturbation} of $u_j$.
Using the estimates \eqref{e:westimates} we find 
\[ [w_j]_l\leq CK^{-1} \nu_j^{l-1}\leq \nu_j^{l-1} \qquad \text{ for }l=0,\ldots, N^{i-1}-J+1,\]
if $K_*$ is large enough depending only on $N^{i-1},n$ and $\rho$. Therefore, estimates \eqref{e:iterativewbounds} are satisfied. Hence,  using \eqref{e:subst2_inductiveinduced} and  the inductive assumption \eqref{e:subst2_iterass},
\begin{align}
    Du_j^TDu_j - \left(Du^T Du + \delta\sum_{k=i}^j a_k^2 \eta_{ik}\otimes \eta_{ik}\right) = \delta \mathcal F_{i}^{j} +\delta  \mathcal R_{i}^j
\end{align}
with \begin{align}
    \mathcal F_i^j &:= \mathcal F_i^{j-1}+\sum_{k=i}^{j-1}\frac{\nu_k}{\nu_j}\mathcal F_{l,k}^J\\
    \mathcal R_i^j &:=  \mathcal R_{i}^{j-1} +R +\sum_{k=i}^{j-1}\left(\left(\tfrac{\nu_k   }{\nu_j}\right)^{J+1}\gamma_k^J A_k^J\odot\eta_{ik}+\tfrac{\nu_k}{\nu_j}\mathcal F_{m,k}^J\right).
\end{align}
It remains to show that $\mathcal F_i^j$ and $\mathcal R_i^j$ satisfy \eqref{e:subst2F_i}, \eqref{e:subst2R_i}, that \eqref{e:susbst2_iterassv-u} is satisfied, and to prove the lemma.

First, since $\mathcal F_{l,k}^J\in \mathcal V_{i+1} $ for any $k=i,\ldots,j-1$, the same holds for $\mathcal F_i^j$ by the inductive assumption for $\mathcal F_i^{j-1}$. Moreover, from the inductive assumption and \eqref{e:F_lestimates} 
\[\|\mathcal F_i^j\|_0\leq CK^{-1}\,,\]
whereas the other estimates in \eqref{e:subst2F_i} follow from \eqref{e:F_lestimates} in view of $\nu_k\leq \nu_j$ for all $k=i,\ldots,j-1$, and the inductive assumption. 

The fact that $\mathcal R_i^j$ satisfies \eqref{e:subst2R_i} follows from the inductive assumption, \eqref{e:inductiveR}, $\nu_j=K\nu_{j-1}$ and \eqref{e:inductivemixedblock}. 

Lastly, from \eqref{e:stepv-u} and the inductive assumption we find
\[\|u_j-u\|_0\leq C\sdel\lambda^{-1}\,,\quad [u_j-u]_{k+1}\leq C\sdel\nu_j^k\] 
for all $k=0,\ldots,\min\{N^{i-1}-J, N^{i}_{j-1}+1-1\}$. But
\[\min\{N^{i-1}-J, N^{i}_{j-1}+1-1\} = N^{i-1}-J-(j-1-i) = N^i_j+1\,,\] i.e., \eqref{e:susbst2_iterassv-u} follows.

It therefore remains to prove the lemma. 

 \begin{proof}[Proof of Lemma \ref{l:inductive}]
Observe first that from \eqref{e:inducedmetric}, we get
\begin{align*}
    Du_j^TDu_j &= Du_{j-1}^TDu_{j-1} + \delta a_j^2 \eta_{ij}\otimes\eta_{ij} + \delta^{\sfrac{3}{2}}\mathcal R + 2\delta\sym(Dw_j)\\
         &\hspace{1cm }  + 2\delta^{\sfrac{1}{2}}\frac{a_j\gamma_2}{\nu}\sym\left( Du_{j-1}^T D\zeta\right) +2\delta \frac{2\gamma_1+\gamma_2\gamma_2'}{\nu}\sym\left( \nabla(a_j^2)\otimes \eta_{ij}\right) +\delta \frac{\gamma_2^2}{\nu^2}\nabla a_j\otimes \nabla a_j\,,
\end{align*}
where $\nu= \nu_j = K^{J+j-i}\lambda$. We will now identify all terms which are already of sufficiently small size $O(\delta(K^{-J}+\sdel))$ to be put into $\delta R$. Under the assumption $[w_j]_l\leq \nu_j^{l-1}$, $\delta^{\sfrac{3}{2}}\mathcal R$ is such a term by \eqref{e:steprest}. The last two terms are also sufficiently small:
\[\|\delta \frac{2\gamma_1+\gamma_2\gamma_2'}{\nu}\sym\left( \nabla(a_j^2)\otimes \eta_{ij}\right)\|_0\leq C\delta \|a_j\|_0[a_j]_1 \nu^{-1} \leq C \delta\frac{\lambda}{\nu}\leq C \delta K ^{-(J+j-i)}\leq C\delta K^{-J}\,,\]
and 
\[ \|\delta\frac{\gamma_2^2}{\nu^2}\nabla a_j\otimes \nabla a_j\|_0\leq C\delta[a_j]_1^2\nu^{-2}\leq C\delta\frac{\lambda^2}{\nu^2}\leq C\delta K^{-J}\,.\]
Therefore, we will only need to analyze
\begin{equation}\label{e:importanterror}
    2\delta^{\sfrac{1}{2}}\frac{a_j\gamma_2(\nu_j x\cdot \eta_{ij})}{\nu_j}\sym\left( Du_{j-1}^T D\zeta\right),
\end{equation}
since this is only of the order $\delta \nu_{j-1}\nu_j^{-1} =\delta K^{-1}$ by \eqref{e:subst2_zeta}. 

The crucial observation is that the large part of this term is in fact of the form 
\[\delta\sum_{k=i}^{j-1}\frac{\nu_k}{\nu_j}\gamma_2(\nu_j x\cdot \eta_{ij}) A_k(\nu_{k} x\cdot\eta_{ik},x)\odot \eta_{ik}\]
for  $A_k\in C^\infty(\S^1\times \bar\Omega,\R^n)$ given by \eqref{d:A_k}. 

Indeed, observe that since $\zeta$ is normal to $u_{j-1}$ it holds 
\begin{equation}\label{e:secondder}
    \sym(Du_{j-1}^T D\zeta) = -  D^2 u_{j-1} \cdot \zeta\,,
\end{equation}
where the latter expression denotes the matrix with entry ${kl}$ equal to $-\partial^2_{kl}u_{j-1}\cdot \zeta$. On the other hand, by construction, setting $u_{i-1}:= u $, 
\begin{align*}D^2 u_{j-1} &= \sum_{k=i}^{j-1}\left(D^2u_k-D^2u_{k-1}\right) +D^2 u\\&= \sum_{k=i}^{j-1}\left(D^2u_k-\left(D^2u_{k-1}+\sdel\nu_k a_k\gamma_2''(\nu_k x\cdot\eta_{ik})\zeta_k\eta_{ik}\otimes\eta_{ik}\right)\right)+D^2 u\\
&\quad +\sum_{k=i}^{j-1}\sdel\nu_k a_k\gamma_2''(\nu_k x\cdot\eta_{ik})\zeta_k\eta_{ik}\otimes\eta_{ik} \\
&= \sum_{k=i}^{j-1}\left(D^2u_k-\left(D^2u_{k-1}+\sdel\nu_k a_k\gamma_2''(\nu_k x\cdot\eta_{ik})\zeta_k\eta_{ik}\otimes\eta_{ik}\right)\right)+D^2 u\\
&\quad +\sum_{k=i}^{j-1}\sdel\nu_k a_k\gamma_2''(\nu_k x\cdot\eta_{ik})(\zeta_k-\zeta)\eta_{ik}\otimes\eta_{ik}\\
&\quad +\sum_{k=i}^{j-1}\sdel\nu_k a_k\gamma_2''(\nu_k x\cdot\eta_{ik})\zeta\eta_{ik}\otimes\eta_{ik}\\
& =: I\\
&\quad +II \\
&\quad + III\,,\end{align*}
so that the large error term splits as 
\[ 2\delta^{\sfrac{1}{2}}\frac{a_j\gamma_2(\nu_j x\cdot \eta_{ij})}{\nu_j}\sym\left( Du_{j-1}^T D\zeta\right) = -2\delta^{\sfrac{1}{2}}\frac{a_j\gamma_2(\nu_j x\cdot \eta_{ij})}{\nu_j}(I+II+III)\cdot \zeta\,.\]
Now observe that by estimate \eqref{e:sharpD2} and by the assumption $\|D^2u\|_0\leq \sdel\lambda$ it follows 
\[\|I\|_0\leq C\sdel(\lambda +\sdel \nu_{j-1}).\] 
Moreover, using the inductive assumption $[u_k-u]_1\leq C\sdel$ for all $k=i,\ldots,j-1$ we find from \eqref{e:normalcloseness} that in fact $\|\zeta_k-\zeta\|_0\leq C\sdel$, so that 
\[ \|II\|_0\leq C\delta\nu_{j-1}\,.\]
This shows that these first two terms are small enough: \[\|2\delta^{\sfrac{1}{2}}\frac{a_j\gamma_2(\nu_j x\cdot \eta_{ij})}{\nu_j}(I+II)\cdot \zeta\|_0\leq C\delta\left(\frac{\lambda}{\nu_j} +\sdel\frac{\nu_{j-1}}{\nu_j}\right)\leq C\delta\left( K^{-J}+\sdel\right)\,.\]
Lastly, using $|\zeta|=1$ yields that 
\[- 2\delta^{\sfrac{1}{2}}\frac{a_j\gamma_2(\nu_j x\cdot \eta_{ij})}{\nu_j}III\cdot\zeta = \delta \sum_{k=i}^{j-1} \frac{\nu_k}{\nu_j}\gamma_2(\nu_jx\cdot\eta_{ij}) A_k(\nu_k x\cdot\eta_{ik},x)\odot \eta_{ik},\]
with $A_k$ from \eqref{d:A_k}, as desired. 
\end{proof}
This finishes the proof of the proposition.\end{proof}

We now prove Proposition \ref{p:stage} by iterating on substages after decomposing the metric defect with the help of Lemma \ref{l:babykaellen}. 

\subsection{Proof of Proposition \ref{p:stage}}
As is apparent in \eqref{e:ajestimate}, \eqref{e:subst1v-u} and \eqref{e:subst2v-u}, when applying the decomposition lemma and each of the substages, we loose information on derivatives: from the control of $N_K$ derivatives on $H$ we extract information on $N_K-J+1$ derivatives of the coefficients $a_{ij}$; from information on $N^0+2$ derivatives on $u$ we extract information on $ N^0-n(J+1) +2$ derivatives of $v$, etc. This so-called "loss of derivatives" is managed by a preliminary mollification of the map $u$ and the metric $g$. It then suffices to choose  large enough $N_K,N^0$ (depending only on $J,n$)  such that information on $N_K+1$ (resp. $N^0+2$) derivatives (of the mollification) of $g$ (resp. $u$) yields a control over the $C^2$ norm of the final map $v$.

As can be checked, it suffices to fix 
\begin{equation}
    \label{d:Ns} N^0 :=2(n-1)J+n(n-1)/2+2\,, \quad N_K :=N^0+J,
\end{equation}
although the precise values are not important, only that they can be chosen large enough depending only on $J,n$.

\textbf{Step 0: Mollification.} Let us therefore mollify the map $u$ and the metric $g$ with a standard, radially symmetric mollifier $\varphi_\ell$ at length scale 
\begin{equation}\label{d:ell}
\ell = \frac{d}{\hat C \lambda}\,,
\end{equation}
where $\hat C\geq 1$ is a constant only depending on $n,J$ and $g$ to be chosen below. The parameter $\ell$ is chosen in such a way as not to increase the metric defect through the mollification process. Let then $u_\ell = u\ast\varphi_\ell, g_\ell =g\ast\varphi_\ell$. Since $\ell\leq d $ it follows that $u_\ell, g_\ell$ are well-defined on $\bar V$ and smooth there. Moreover, the mollification estimates from Lemma \ref{l:mollification} yield 
\begin{equation}\label{e:u-uell}
      \| u- u_\ell\|_k \leq C \ell^{2-k}\|u\|_2 \leq C \sdel \lambda^{k-1}\text{ for } k =0,1,
\end{equation}
for a constant $C$ only depending on $n$, and
\begin{equation}\label{e:g-gell}
    \|g-g_\ell\|_0 \leq C\ell^2\|g\|_2 \leq \lambda^{-2}
\end{equation}
if $\hat C$ is large enough depending on $g$. 

Moreover, for any $k=1,\ldots,N^0+1$ we find 
\begin{equation}\label{e:uellestimates}
    [ u_\ell]_{k+1}\leq C \ell^{1-k}[u]_2\leq \sdel \ell^{-k}
\end{equation}
if $\hat C$ is large enough depending on $n,N^0$, i.e., $n$ and $J$.

\textbf{Step 1: Decomposition.} Define the matrix field 
\begin{equation}\label{d:H}
H = \frac{1}{\delta}\left( g_\ell - Du_\ell^TDu_\ell - \hat \delta H_0\right)\,.
\end{equation} 
By \eqref{a:metricdeficit} it follows with the help of \eqref{e:quadraticmoll}
\begin{align*}
    \|H-H_0\|_0&\leq \frac{1}{\delta}\left(\|(g-Du^TDu-\delta H_0)\ast\varphi_\ell \|_0 + \|Du^TDu-Du_\ell^TDu_\ell\|_0 + \hat\delta |H_0|\right)\\
    &\leq 2r +C\ell^2[u]_2^2 \leq \frac{r_K}{2}
\end{align*}
if $r_* $ is small enough and $\hat C$ is large enough depending on $r_K(n,N_K)$ and $n$, i.e., $n $ and $J$. Here and below, $r_K$ is the constant from Lemma \ref{l:babykaellen}.

Moreover, we can similarly estimate for $k=1,\ldots, N_K+1$,
\begin{align*}
    [H]_k &\leq \frac{C_{N_K}}{\delta}\left(\left[(g-Du^TDu)\ast\varphi_\ell\right]_k+ \left[Du^TDu-Du_\ell^TDu_\ell\right]_k\right)\\
    &\leq C_{N_K}\ell^{-k}.
\end{align*}
We absorb the latter constant by setting $ \lambda_0:= C_{N_K}\ell^{-1}$ so that $[H]_k\leq \lambda_0^k$ for all $k=1,\ldots, N_K+1$. 

Now fix 
\begin{equation}\label{d:firstnfreq}
    \tilde \lambda_0:= \tilde C \lambda_0\,, \quad \lambda_i= \tilde \lambda_0 K^i  \text{ for } i =1,\ldots, n\,,
\end{equation}
for a constant $\tilde C$ depending only on $J,n$ to be chosen below. We assume moreover that $K$ is large enough (depending on $n$ and $J$) such that 
\[\frac{\lambda_0}{\lambda_1}=(\tilde C K) ^{-1}\leq K^{-1}\leq \frac{r_K}{2}\,.\]
We can then apply Lemma \ref{l:babykaellen} to find coefficients $\{a_{ij}\}_{1\leq i\leq j\leq n}\subset C^\infty(\bar V)$ and an error term $\mathcal E\in C^\infty(\bar V,\Sym_n)$ such that
\begin{equation}\label{e:initialdecomp}
    H = \sum_{1\leq i\leq j\leq n} a_{ij}^2 \eta_{ij}\otimes\eta_{ij} + \sum_{k=1}^n\frac{2\pi}{\lambda_k^2} \nabla a_{1k}\otimes\nabla a_{1k} +\mathcal E\,,
\end{equation}
with 
\begin{equation}\label{e:Kaellenerror}
    \|\mathcal E\|_0 \leq C K^{-J}
\end{equation}
and
\begin{equation}
    a_{ij}\geq r_K\,, \quad [a_{ij}]_k\leq C\lambda_0^k\text{ for } k=0,\ldots, N_K-J+1 \,.
\end{equation}
If $\tilde C$ is large enough, depending on $n,J$, then we have in fact 
\begin{equation}\label{e:firstaest}
    [a_{ij}]_k\leq \tilde\lambda_0^k \text{ for all } k=1,\ldots,N_K-J+1=N^0+1\,.
\end{equation} 

\textbf{Step 2: property $(P_\rho)$.}
Observe that assumption \eqref{a:metricdeficit} implies 
\[ Du^T Du = g-\delta H_0 -\left(g- Du^TDu -\delta H_0\right) \leq g-\delta H_0 +r\delta \mathrm{Id} = g-\delta\left(H_0-r\mathrm{Id}\right)\leq g \]
if $r\leq 1$ is small enough depending on $n$. Similarly, 
\[Du^T Du \geq \frac12 g\]
if $\delta$ is small enough depending on $n $ and $g$. 
Moreover, taking the trace yields 
\begin{equation}
    \|Du\|_0\leq C(g)\,.
\end{equation}
Since
\[
    Du_\ell^TDu_\ell = Du^TDu +Du^T(Du_\ell-Du) + (Du_\ell-Du)^TDu_\ell ,
\]
it follows by  \eqref{e:u-uell}, that  
\[\frac14 g\leq  Du_\ell^T Du_\ell \leq 2g\,, \]
if $\delta $ is small enough depending on $n$ and $g$. Thus, in this case we have
\begin{equation}\label{e:uellprho}
    \frac{1}{\rho}\mathrm{Id}\leq Du_\ell^T Du_\ell \leq \rho \mathrm{Id}
\end{equation}
for some $\rho$ only depending on $n$ and $g$. If necessary, we can enlarge $\rho$,  depending additionally on $J$, such that in addition we have for all $1\leq i\leq j\leq n$
\begin{equation}
    a_{ij}\leq \rho.
\end{equation}

\textbf{Step 3: adding primitive metrics} We now iteratively add the families $i=1,\ldots,n$ of primitive metrics by applying first Proposition \ref{p:substage1} and then Proposition \ref{p:substage2} successively. The following lemma captures the inductive procedure. 

\begin{lemma}\label{l:iterative} 
If $K$ is large enough (depending on $n,J,\rho$) and $\delta$ small enough (depending on $n,J,\rho,g$), then for $i=1,\ldots,n$ there exists $u_i\in C^\infty(\bar V,\R^{n+1})$ satisfying $(P_{\rho_i})$ with $\rho_i=(1+i/n)\rho$ and
\begin{equation}\label{e:ithdefect}
     Du_i^TDu_i = Du_\ell^TDu_\ell + \delta\sum_{k=1}^i \sum_{l=k}^na_{kl}^2 \eta_{kl}\otimes\eta_{kl} +\delta \sum_{k=1}^n \frac{2\pi}{(\lambda_k)^2}\nabla a_{1k}\otimes \nabla a_{1k} + \delta \mathcal F^i + \delta \mathcal R^i\,,
\end{equation}
for $\mathcal F^i, \mathcal R^i \in C^\infty(\bar V,\Sym_n)$ satisfying
\begin{align}
     &\mathcal F^i\in \mathcal V_{i+1}\,, \quad \|\mathcal F^i\|_0\leq CK^{-1},\quad [\mathcal F^i ]_k\leq C\mu_i^k \quad \text{ for } k=1,\ldots,N_i+1\,,\label{e:Fi}\\
     &\|\mathcal R^i\|_0\leq C\left( K^{-J}+\sdel\right)\,,\label{e:Ri}\\
     &\|u_i-u_\ell\|_0\leq C\sdel\tilde\lambda_0^{-1}\,,\quad [u_i-u_\ell]_{k+1}\leq C\sdel\mu_i^k\quad\text{ for } k=0,\ldots,N_i+1,\label{e:ui-uell}
\end{align}
where $\mu_i= \bar C_i K^{(i-1)J+ni}\tilde \lambda_0$ for constants $\bar C_i$ depending only on $n,g,J,i$, and  \begin{equation}\label{d:N^i}
   \begin{split}
        N_1&= N^0-n(J+1)\\
      N_i &= N_1-(i-1)J-(i-1)n+ i(i+1)/2-1\quad \text { for }2\leq i\leq n\\
      N_n & = N_{n-1}-1.
   \end{split}
   \end{equation}
\end{lemma}

In particular, once the lemma is proven, we see that the map $v:=u_n$ has the desired properties. Indeed, recalling $\mathcal V_{n+1}=\{0\}$, the decomposition \eqref{e:initialdecomp} and the definition of $H$ in \eqref{d:H}, we find that 
\begin{align*}
         Du_n^TDu_n &= Du_\ell^TDu_\ell + \delta\sum_{1\leq k\leq l\leq n}a_{kl}^2 \eta_{kl}\otimes\eta_{kl} +\delta \sum_{k=1}^n \frac{2\pi}{(\lambda_k)^2}\nabla a_{1k}\otimes \nabla a_{1k} +  \delta \mathcal R^n\,\\
         & = Du_\ell^TDu_\ell+\delta( H -\mathcal E)+ \delta\mathcal R^n\\
         &= g_\ell -\hat\delta H_0 +\delta(\mathcal R^n-\mathcal E)\\
         &= g -\hat \delta H_0 + \delta(\mathcal R^n-\mathcal E) + (g_\ell-g)\,.
\end{align*}
The estimate \eqref{e:metricdeficit} therefore follows in view of \eqref{e:Ri}, \eqref{e:Kaellenerror} and \eqref{e:g-gell}, whereas estimates \eqref{e:C1} and \eqref{e:C2} follow from \eqref{e:ui-uell} and \eqref{e:u-uell} in view of $N_n=0$ 
% \begin{align*}
%     N_n&=N_1-(n-2)(J+n)+n(n-1)/2-1 -1\\
%     &= N^0-nJ-n -(n-2)(J+n)+n(n-1)/2-2 \\
%     &= N^0-2(n-1)J-n(n-1)/2-2 \\
%     &= 0
% \end{align*}
by definition of $N^0 $ in \eqref{d:Ns}, and 
\[\mu_n = \bar C_n K^{(n-1)J + n^2}\tilde\lambda_0 = \bar C_n\tilde C C_{N_K} \ell^{-1}K^{(n-1)J + n^2} = C d^{-1} \lambda K^{(n-1)J + n^2} \]
for some constant $C$ only depending on $n,J, g$.
Hence, it suffices to prove the lemma. 
\begin{proof}[Proof of Lemma \ref{l:iterative}]
\textbf{$i=1$: Adding the first family.} We start the finite induction by using Proposition \ref{p:substage1} to verify $i=1$. 

Recall \eqref{e:firstaest}, \eqref{e:uellestimates}, \eqref{e:uellprho}.  Assuming our $K\geq K_*(n,N^0,\rho)$, we can therefore apply Proposition \ref{p:substage1} with \[u=u_\ell,\quad  a_j= a_{1j} \text{  for }j=1,\ldots,n,\quad \lambda= \tilde \lambda_0\]
to generate a map $u_1:=v$, which, satisfies \eqref{e:ithdefect}–\eqref{e:ui-uell} for $i=1$
by \eqref{e:subst1induced}–\eqref{e:subst1v-u} upon setting $\mathcal F^1=\mathcal F_1, \mathcal R^1 =\mathcal R_1$, where we used that $\lambda_k = K^k\tilde \lambda_0$. 
\medskip

\textbf{$i\geq 2$: Adding family $i$.} Assume therefore that the claim is true up to some $1\leq i-1\leq n-1$. We will now apply Proposition \ref{p:substage2} to add family $i$ of primitive metrics. 

First, $u_{i-1}$ satisfies $(P_{2\rho})$ by the inductive assumption. 

Next, since $\mathcal F^{i-1}\in \mathcal V_i$ we can decompose 
\begin{equation}
    \label{e:Fidecomp}
    \mathcal F^{i-1} = \sum_{i\leq k\leq l\leq n} L_{kl}(\mathcal F^{i-1})\eta_{kl}\otimes\eta_{kl}.
\end{equation} 
We now define the adjusted amplitudes 
\begin{equation}
    b_{kl}= \sqrt{a_{kl}^2-L_{kl}(\mathcal F^{i-1})}\,.
\end{equation}
for $i\leq k\leq l\leq n$. These are well-defined for large enough $K$, since by \eqref{e:Fi} and by linearity of the maps $L_{kl}$ it holds 
\[ \|L_{kl}(\mathcal F^{i-1})\|_0\leq CK^{-1}\leq \frac{r_0^2}{2}\]
if $K$ is large enough depending on $n,J,\rho(n,J)$, and $a_{kl}\geq r_0$ by \eqref{e:ajlower}.  
Moreover,
\[ [L_{kl}(\mathcal F^{i-1})]_m\leq C \mu_{i-1}^m\]
for $m=1,\ldots,N_{i-1}+1$, which implies 
\begin{equation}\label{e:bestimate}
b_{kl}\leq 2\rho \,,\quad [b_{kl}]_m\leq C(\tilde\lambda_0^m+\mu_{i-1}^m) \leq C\mu_{i-1}^m
\end{equation}
for all $m=1,\ldots,N_{i-1}+1$, where we used Lemma \ref{l:composition} and the lower bound $a_{kl}^2-L_{kl}(\mathcal F^{i-1})\geq r_0^2/2$ in the last estimates. 

We absorb the constants by setting $\tilde \lambda = \tilde C \mu_{i-1}$ for some constant $\tilde C$ only depending on $n$ and $J$ to find that 
\[ [u_{i-1}]_{k+1}\leq \sdel \tilde\lambda^k\,,\quad [b_{il}]_k\leq \tilde\lambda^k\]
for all $k=1,\ldots,N_{i-1}+1$ and $l=i,\ldots,n$. We can therefore apply Proposition \ref{p:substage2} with 
\[ u= u_{i-1},\quad  a_j = b_{ij} \text{ for } j=i,\ldots,n,\quad \lambda=\tilde\lambda, \quad N^{i-1}=N_{i-1}\]
to generate $v=:u_i$ such that 
\[ Du_i^TDu_i -\left(Du_{i-1}^TDu_{i-1} +\delta \sum_{j=i}^n \left(a_{ij}^2-L_{ij}(\mathcal F^{i-1})\right)\eta_{ij}\otimes\eta_{ij} \right) = \delta \mathcal F_i +\delta \mathcal R_i \]
for some $\mathcal F_i,\mathcal R_i$ satisfying \eqref{e:subst2F_i},\eqref{e:subst2R_i}. Moreover, we have 
\begin{equation}
    \label{e:u2-u1}
    \begin{split}
        \|u_i-u_{i-1}\|_0&\leq C\sdel \tilde \lambda^{-1}\,,\\
        [u_i-u_{i-1}]_{k+1}&\leq C\sdel \left( \tilde \lambda K^{J+n-i}\right)^k\quad \text{ for }k=0,\ldots,N^i+1\,,
    \end{split}
\end{equation}
where we recall from \eqref{d:N^i} that
\[N^i = \begin{cases}
            N^{i-1}-1  \quad \text{ if } i=n\\
            N^{i-1 }- J-(n-i) \quad \text{ for } 2\leq i \leq n-1\end{cases}\]
            and $N^{i-1}=N_{i-1}$. In particular, one verifies by induction that $N^i=N_i$.
            
            % if $i=n$ we have $N^n = N^{n-1}-1=N_{n-1}-1=N_n$, and if $i\leq n-1$ we have, using the inductive hypothesis, 
            % \begin{align*}
            %     N^i = N_{i-1}-J-(n-i) &= (N_1-(i-2)J-(i-2)n + i(i-1)/2-1) -J -(n-i)\\
            %     & = N_1-(i-1)J-(i-1)n+ i(i+1)/2 -1 = N_i\,.
            % \end{align*}
            Thus, by the inductive assumption and \eqref{e:subst2v-u}, we find 
\[\| u_i-u_\ell\|_0\leq C\sdel\left(\tilde\lambda_0^{-1}+ \tilde\lambda^{-1}\right)\leq C\sdel \tilde\lambda_0^{-1}\,,\]
 as well as 
\[[ u_i-u_\ell]_{k+1}\leq C\sdel\left( \mu_{i-1}^k+ (\tilde \lambda K^{J+n-i})^k\right)\]
for all $k=0,\ldots,N_i+1$, i.e. \eqref{e:ui-uell} holds, since
\begin{align*}\tilde\lambda K^{J+n-i} = \tilde C\mu_{i-1}K^{J+n-i} &= \tilde C\bar C_{i-1} \tilde\lambda_0 K^{(i-2)J+n(i-1)}K^{J+n-i} \\ 
&= \tilde C\bar C_{i-1}\tilde\lambda_0 K^{(i-1)J+ni-i}\leq \tilde C\bar C_{i-1}\tilde\lambda_0 K^{(i-1)J+ni} =\mu_i\,.\end{align*}
From the inductive assumption \eqref{e:ithdefect} and the decomposition \eqref{e:Fidecomp}  it follows that 
\begin{align*}
Du_{i-1}^TDu_{i-1} + \delta \sum_{j=i}^n \left(a_{ij}^2-L_{ij}(\mathcal F^{i-1})\right)\eta_{ij}\otimes\eta_{ij} &= Du_\ell^TDu_\ell + \delta \sum_{k=1}^i\sum_{j=k}^n a_{kj}^2\eta_{kj}\otimes\eta_{kj}  \\
&\quad +\delta \sum_{k=1}^n \frac{2\pi}{(\lambda_k)^2}\nabla a_{1k}\otimes \nabla a_{1k}   + \delta \mathcal R^{i-1}\\&\quad + \delta \sum_{i+1\leq k\leq l\leq n} L_{kl}(\mathcal F^{i-1})\eta_{kl}\otimes\eta_{kl}\,.
\end{align*} 
We now set 
\begin{align*}
   & \mathcal F^i = \mathcal F_i+ \sum_{i+1\leq k\leq l\leq n} L_{kl}(\mathcal F^{i-1})\eta_{kl}\otimes\eta_{kl} \\
    &\mathcal R^i = \mathcal R_i + \mathcal R^{i-1},
\end{align*}
so that the decomposition \eqref{e:ithdefect} holds. 

Moreover, it follows $\mathcal F^i \in \mathcal V_{i+1}$ with 
\[\|\mathcal F^i\|_0\leq CK^{-1}\]
and for all $k=1,\ldots,N_i+1$ we have, using the inductive assumption and \eqref{e:subst2F_i},
\[ [\mathcal F^i]_k\leq C\left( (\tilde \lambda K^{J+n-i})^k+ \mu_{i-1}^k\right) \leq C \mu_i^k\]
by definition of $\mu_i$ and $K\geq 1$. Thus, \eqref{e:Fi} holds. 

On the other hand, \eqref{e:Ri} follows from the inductive assumption and \eqref{e:subst2R_i}. 

Lastly, the fact that $u_i $ satisfies $(P_{\rho_i})$ for small enough $\delta$ follows from the inductive assumption and \eqref{e:subst2v-u}. This concludes the proof.
\end{proof}

\section{Proof of Theorem \ref{t:main}}\label{s:pfofmain}
The proof of the main theorem follows by iteratively applying Proposition \ref{p:stage} to obtain a sequence of immersions $\{u_q\}_{q\in \N}$ which converges in the desired H\"older space to an immersion. It is analogous to the proof of the main theorem in \cite{CaoHirschInauen25}. For the reader's convenience we repeat the proof here. 

In order to apply Proposition \ref{p:stage}, we first need to construct a short immersion $u_0$ satisfying the conditions in Proposition \ref{p:stage}, since $\underline u$ need not satisfy either \eqref{a:metricdeficit} or \eqref{a:C2}. We achieve this through one stage of a classical Nash--Kuiper iteration in the next subsection. Once $u_0$ is constructed, we then iterate Proposition \ref{p:stage} to generate the sequence $\{u_q\}$ in Subsection \ref{s:iter}, and finally show in Subsection \ref{s:convergence} that the sequence converges in $C^{1,\theta}$ to an isometric immersion.
\subsection{Initial short immersion} \label{s:initial}

Due to the mollification step at the beginning of each stage proposition, the maps $\{u_q\}$ will be defined on smaller and smaller domains. In order to end up with an isometry on $\Omega$ we first extend both the metric $g$ and the short map $\underline u$ to $\R^n$ such that
\[ \|\underline u\|_{C^1(\R^n)} \leq C\|\underline u\|_{C^1(\bar \Omega)}\,, \quad \| g\|_{C^2(\R^n)} \leq C\|g\|_{C^2(\bar \Omega)}\,,\]
for a constant only depending on $n$ and $\Omega$. Such an extension procedure is well-known (see e.g. \cite{Whitney}).  By mollification and scaling  we can also assume that $\underline u$ is smooth and strictly short for $g$ on $\bar \Omega$.

By continuity we can find an open and simply connected set $V_0 \Subset \R^n$ depending on $\underline u$ and $g$, and $0<\underline\delta$ such that 
\begin{align}
    &\Omega \Subset V_0,\nonumber\\
    & g  - D\underline u^TD\underline u> 2\underline\delta h_0 \text{ on } \bar V_0\,,\label{e:short1}\\
    &\underline u \text{ is an immersion on } \bar V_0\,.\nonumber
\end{align}
By \eqref{e:short1}, for any $0\leq \delta_0<\underline\delta$, we have
\begin{equation}\label{e:firstlowerbound}
    g-D\underline u^TD\underline u - \delta_0 h_0 \geq\underline\delta h_0.
\end{equation}
We can therefore decompose 
\[ g-D\underline u^TD\underline u - \delta_0 h_0  = \sum_{i=1}^M a_i^2 \eta_i\otimes \eta_i \] 
for some $M\in \N$, unit vectors $\eta_i \in \S^{n-1}$, and $a_i \in C^2(\bar V_0)$ (see for example \cite[Lemma 3.3]{Laszlo}). Notice that while $M$ might depend on $g,\underline u$ and $\underline\delta$, it can be chosen independently of $\delta_0$ due to \eqref{e:firstlowerbound}. Following the standard Nash--Kuiper iteration procedure with classical Kuiper-type corrugations (as in \cite{CDS}), one obtains for any $\mu\geq \mu_0(\underline u, g)$ a smooth immersion $u_0$ and metric error $\mathcal{E}_0$ satisfying
\[Du_0^TDu_0=D\underline u^TD\underline u+\sum_{i=1}^M a_i^2 \eta_i\otimes \eta_i+\mathcal{E}_0\]
with
\begin{align}\label{e:u0-estimate}
    \|u_0-\underline u\|_0\leq \frac{C_1}{\mu},\quad \|u_0\|_2\leq C_1\mu^{M},\quad \|\mathcal{E}_0\|_0\leq \frac{C_1}{\mu}\,,
\end{align}
for a constant $C_1$ depending on $\underline u, g, M$ and $\underline \delta$, but not on $\delta_0$. While similar to the arguments presented earlier in this paper, the construction of $u_0$ uses classical corrugations rather than our modified ansatz, since at this stage the metric defect need not be small (a smallness assumption required for our approach, cf. Remark \ref{r:osc}). For a detailed proof, see \cite[Proposition 3.2]{CaoSze2019}, which can be applied directly to generate $u_0$. 

Thus we have
\begin{equation}
    \label{e:u0-deficit}
    g-Du_0^TDu_0-\delta_0 h_0=-\mathcal{E}_0\,.
\end{equation}
We now show that for a suitable choice of $\delta_0, \lambda_0$ and $\mu\geq \mu_0(\underline u,g)$, $u_0$ satisfies the induction assumptions. For this,   take
\[\delta_0:=a^{-\tau}\,,\, \quad \lambda_0:=a^{(M+1)\tau}\]
with some constant $1>\tau>0$ and large constant $a$ to be fixed below. For $r\leq 1$ to be fixed below in Subsection \ref{s:iter}, we choose
\begin{equation}\label{d:mu}
     \mu:= \max\{C_1r^{-1}a^{\tau}, 2C_1\varepsilon^{-1}, \mu_0(\underline u, g)\}=C_1r^{-1}a^{\tau},
\end{equation}
provided that $a$ is large. In particular, with this choice we have
\begin{equation}\label{e:delta0}
\frac{C_1}{\mu}\leq\min\{\frac\varepsilon2,\, r\delta_0\}.
\end{equation}
By taking $a$ larger if necessary (depending on $\underline u, g, r, \tau, \varepsilon$ and $M$) we can ensure that
\begin{equation}
    \label{e:lambda0}
\delta_0^{\sfrac12}\lambda_0=a^{(M+\frac12)\tau}\geq C_1^{M+1} r^{-M}a^{M\tau}=C_1\mu^M.
\end{equation}
Hence from \eqref{e:u0-deficit}, \eqref{e:delta0} and \eqref{e:lambda0}, we get the initial short immersion $u_0\in C^\infty(\bar V_0, \R^{n+1})$ such that
\begin{equation}\label{e:u0-induction}
\begin{split}
    &\|u_0-\underline u\|_0\leq\frac{\varepsilon}{2}, \quad \|u_0\|_2\leq\delta_0^{\sfrac{1}{2}}\lambda_0,\\
    &\|g-Du_0^TDu_0-\delta_0 h_0\|_0\leq r\delta_0.
\end{split} 
\end{equation}

\subsection{Iteration}\label{s:iter} We now want to iteratively apply Proposition \ref{p:stage} to generate the sequence $\{u_q\}$ converging to our desired immersion. For this, we first need to fix our parameters and start with fixing  any desired H\"older exponent 
\[ \theta <\frac{1}{1+2(n-1)}.\]
We now choose $J\in \N$ large enough such that when we iteratively apply Proposition \ref{p:stage} with this choice, our sequence of immersions converges in $C^{1,\theta}$. As we will see, choosing $J\in \N$ so large such that 
\begin{equation}\label{e:Jexponent}
    \beta:=\frac{1}{1+2(n-1) +4n^2/J}\in\,]\theta,\,\frac{1}{1+2(n-1)}\,[ \,
\end{equation}  
will suffice. 

To manage the loss of domain due to mollification in each application of the stage, we choose a sequence $\{V_q\}_{q\in \N}$ of smooth bounded, open domains such that $\Omega \Subset  V_{q+1} \Subset V_q$ for any $q\geq0$,  $\bigcap_q V_q =  \bar\Omega $ and 
\begin{equation}\label{e:distance}
    \mathrm{dist}(V_{q}, \partial V_{q-1}) \geq 2^{-q}\mathrm{dist}(\Omega,\partial V_0) =:d_q \text{ for }q\geq1\,.
\end{equation} 
Define then
\begin{equation}
    \label{e:deltalambda}
    \delta_{q}=\delta_0a^{1-b^q},\quad \lambda_{q}=\lambda_0 a^{\frac{1}{2\beta}(b^q-1)}
\end{equation}
for some $1<b<1+\frac\tau2$ to be chosen in the next section. With this choice, set
\begin{equation}
    \label{e:Lambda}
    K_q=\Lambda(\delta_{q}\delta_{q+1}^{-1})^{\frac1J}
\end{equation}
with constant $\Lambda$ to be fixed. 

\textbf{Claim:} If $a>a_0(\underline u, g, r, n, J, b, M,\tau,\Lambda,  \mathrm{dist}(\Omega,\partial V_0))$ is  sufficiently large, then there exists a sequence of smooth immersions $\{u_q\}_{q\in \N_0}$ 
such that for all $q\geq 0$
\begin{align}
    &u_q\in C^2(\bar V_q,\R^{n+1})\,, \label{e:qass0}\\
    &\|g- Du_q^TDu_q- \delta_q h_0\|_0\leq r\delta_q\,, \label{e:qass1}\\
    &\|u_q\|_2\leq \delta_q^{\sfrac{1}{2}}\lambda_q\,, \label{e:qass3}\\
    &\|u_{q+1}- u_{q} \|_k\leq C\delta_{q}^{\sfrac12}\lambda_q^{-k}, \text{ for } k=0, 1,\label{e:qass4} 
\end{align}
where the  constant $C$ depends on $n, g, \underline u.$

Indeed, by \eqref{e:u0-induction}, $u_0$ from Subsection \ref{s:initial} satisfies the required conditions if we choose $r\leq r_*(n,J,g)$ in the definition of $\mu$ in \eqref{d:mu}. Here, $r_*(n,J,g)$ is the constant in Proposition \ref{p:stage} with $J$ defined as in \eqref{e:Jexponent}. Now   assume $u_q$ has already been constructed. Taking $a$ large we have
 \[\delta_{q+1}\leq \frac{r\delta_q}{|H_0|}\,.\] 
Upon choosing the constant $\Lambda$ in \eqref{e:Lambda} so large that $\Lambda >c_*$ with $c_*$ from Proposition \ref{p:stage} we can apply  Proposition \ref{p:stage}
with
 \[ U=V_q,\, V=V_{q+1},\, u=u_q\,,\, \delta = \delta_q\,,\, \hat\delta = \delta_{q+1}\,,\,\lambda=\lambda_{q}, \,d=d_{q+1},\, K=K_q\] 
 to get $v\in C^2(\bar V_{q+1}, \R^{n+1})$. Set
 \[u_{q+1}:=v\in C^2(\bar V_q, \R^{n+1}).\]
 Then \eqref{e:qass0} holds for $q+1$. To show \eqref{e:qass1} for $q+1$,  using  \eqref{e:metricdeficit}, \eqref{e:deltalambda}, \eqref{e:Lambda} and $b<1+\frac\tau2$, we get 
 \begin{align*}
     \|g- Du_{q+1}^tDu_{q+1}- \delta_{q+1} h_0\|_0\leq& C\delta_q(K_q^{-J}+\delta_q^{\sfrac12})+\lambda_q^{-2}\\
     \leq & C\delta_{q+1}\big(\Lambda^{-J}+\delta_{q+1}^{-1}\delta_q^{\sfrac32}+\delta_{q+1}^{-1}\lambda_q^{-2}\big)\\
     \leq &C\delta_{q+1}\big(\Lambda^{-J}+a^{b-1-\frac\tau2}+a^{-2N\tau-\tau-1+b}\big)\\
     \leq& r\delta_{q+1}\,,
 \end{align*}
after taking $\Lambda\geq 3Cr^{-1/J}$ and $a$ larger. By \eqref{e:C2}, we have
 \begin{align*}
     \|u_{q+1}\|_2\leq& C\delta_{q}^{\sfrac12}\lambda_qd_{q+1}^{-1}K_q^{J(n-1)+n^2}=\delta_{q+1}^{\sfrac12}\lambda_{q+1}A
 \end{align*}
with
\begin{align*}
    A:=& C\frac{\delta_{q}^{\sfrac12}\lambda_q}{\delta_{q+1}^{\sfrac12}\lambda_{q+1}}d_{q+1}^{-1}K_q^{J(n-1)+n^2}\\
     \leq&Ca^{(\frac12-\frac1{2\beta})(b-1)b^q+(b-1)b^q(n-1+\frac{n^2}{J})}\Lambda^{J(n-1)+n^2}2^{q+1}\mathrm{dist}(\Omega,\partial V_0)^{-1}\,.
\end{align*}
By \eqref{e:Jexponent}, we have
\[\frac1{2\beta}-\frac12=n-1+\frac{2n^2}{J}\,,\]
so that 
\[A\leq Ca^{-\frac{n^2}{J}(b-1)b^q}\Lambda^{J(n-1)+n^2}2^q\mathrm{dist}(\Omega,\partial V_0)^{-1}\leq 1\]
after taking $a\geq a_0(\underline u, g, r, n, J, b, \tau,\Lambda , \mathrm{dist}(\Omega,\partial V_0))$ large. Thus we get \eqref{e:qass3}. Moreover, \eqref{e:qass4} follows directly from \eqref{e:C1}. This proves the claim.

\subsection{Convergence}\label{s:convergence} Finally, we show that the sequence constructed in the previous section converges in $C^{1,\theta}(\bar\Omega, \R^{n+1}) $ to an isometric immersion on $\bar \Omega$. By interpolation and $\delta_{q}^{\sfrac12}\lambda_{q}\leq \delta_{q+1}^{\sfrac12}\lambda_{q+1}$, it holds for any $q$
\begin{align*} \|u_{q+1}-u_{q}\|_{1,\theta} &\leq C_\theta \|u_{q+1}-u_{q}\|_1^{1-\theta}\|u_{q+1}-u_{q}\|_2^\theta\\
&\leq C\delta_q^{\frac12(1-\theta)}(\delta_{q+1}^{\sfrac12}\lambda_{q+1})^\theta\\
&=C\delta_0^{\sfrac12}\lambda_0^\theta a^{\frac12-\frac{\theta}{2\beta}}a^{\frac12b^q(\frac{\theta}{\beta}{b}-1-(b-1)\theta)}\,.
\end{align*}
Since $\theta<\beta$, we can take 
\[1<b<\min\left\{1+\frac\tau 2,\, \frac{\beta-\theta\beta}{\theta-\theta\beta}\right\}=1+\min\left\{\frac\tau2,\,\frac{\beta-\theta}{\theta-\theta\beta}\right\}\]
and also $a$ large such that 
\[\|u_{q+1}-u_{q}\|_{1,\theta}\leq \left(\delta_0^{\sfrac12}\lambda_0^\theta a^{\frac12-\frac{\theta}{2\beta}}\right) 2^{-q-1},
\]
and then $\{u_q\}$ is a Cauchy sequence in $C^{1,\theta}(\bar\Omega, \R^{n+1})$. Let $u\in C^{1,\theta}(\bar\Omega, \R^{n+1})$ be the limit. From \eqref{e:qass1} and $\delta_q\to0$ as $q\to\infty,$ we have
\[g-Du^TDu=\lim_{q\to\infty}(g-Du_q^TDu_q)=0.\]
Thus $u$ is an isometric immersion for $g$ on $\bar\Omega.$ Furthermore, by \eqref{e:qass4} for $k=0$, we have
\begin{align*}
    \|u-u_0\|_0\leq& \sum_{q=0}^\infty\|u_{q+1}-u_q\|_0\leq C\sum_{q=0}^\infty\delta_q^{\sfrac12}\lambda_q^{-1}\\
    \leq& C\delta_0^{\sfrac12}\lambda_0^{-1}a^{\frac12+\frac1{2\beta}}\sum_{q=0}^\infty a^{-b^q(\frac12+\frac1{2\beta})}\\
    \leq& C a^{-\frac12\tau-(M+1)\tau}a^{\frac12+\frac1{2\beta}}\sum_{q=0}^\infty a^{-(q+1)(\frac12+\frac1{2\beta})}\\
    \leq& Ca^{-\frac12\tau-(M+1)\tau}\leq\frac\varepsilon2,
\end{align*}
by taking $a$ sufficiently large.  The above inequality together with \eqref{e:u0-induction} implies
\[\|u-\underline u\|_0\leq\|u-u_0\|_0+\|u_0-\underline u\|_0\leq\varepsilon\,.\]
Therefore, the proof is complete.
\newpage

\begin{appendix}
\section{H\"older spaces: interpolation and regularization} Let $\Omega\subset\R^n$ be a bounded open set and $m\in \N$. For a function $f:\Omega\to \R^m$ we define the (semi-)norms
\[\|f\|_0 = [f]_0= \sup_{x\in\Omega} |f(x)|\,,\qquad  [f]_k =\max_{|\sigma|=k}\|D^\sigma f\|_0\,, \text{ for }k\in \N, \]
where for a multi-index $\sigma=(\sigma_1,\ldots,\sigma_n)\in \N_0^{n}$, $D^
\sigma f = \frac{\partial^{\sigma_1}}{\partial x_1}\ldots\frac{\partial^{\sigma_n}}{\partial x_n} f$ and $|\sigma| = \sigma_1+\ldots+\sigma_n$. For $\alpha\in ]0,1]$, the H\"older semi-norms are defined as 
\[[f]_{0,\alpha} = \sup_{\substack{x,y\in \Omega\\ x\neq y}}\frac{|f(x)-f(y)|}{|x-y|^\alpha}\,,\qquad  [f]_{k,\alpha} =\max_{|\sigma|=k} [D^\sigma f]_{0,\alpha}\,. \] 
We sometimes  abbreviate $[f]_{k+\alpha}:=[f]_{k,\alpha}$.
For $k\in \N_0, \alpha\in ]0,1]$ the H\"older norm is given by 
\[ \|f\|_{k,\alpha} = \sum_{j=0}^k[f]_j + [f]_{k,\alpha}\,. \] As usual, $C^{k,\alpha}(\bar \Omega, \R^m)$ denotes the space of $k$-times continuously differentiable functions $f:\Omega\to \R^m$ such that $\|f\|_{k,\alpha}<\infty$. 

We recall the classical interpolation inequalities for the H\"older seminorms and $0\leq r\leq s\leq t$: 
\[ [f]_{s}\leq C[f]_{r}^{\frac{t-s}{t-r}}[f]_t^{\frac{s-r}{t-r}}\,\]
for a constant $C$ only depending on $r,s,t, n,m$ and $\Omega$. Combining these inequalities with the product rule yields the following Leibniz-rule (see \cite{DIS} for a proof): 

If $f_1, f_2\in C^{k, \alpha}(\overline\Omega)$ for any $k\in\mathbb{N}_0$ and $\alpha\in[0, 1]$, then we have
\begin{equation}\label{e:Leibniz}
[f_1f_2]_{k,\alpha}\leq C(\|f_1\|_0[f_2]_{k,\alpha}+[f_1]_{k,\alpha}\|f_2\|_0),
\end{equation}
where the constant only depends on $k,\alpha$ and $\Omega$. 
We also reacll the following lemma for compositions of H\"older maps (see e.g. \cite{DIS}).
    \begin{lemma}
    \label{l:composition}
    Let $U\subset \R^n, V\subset \R^m$ be bounded open sets and $ f:U\to V$ and $F:V\to\R$ be two smooth functions. Then, for every natural number $l>0$ , there is a constant $C$
 (depending only on $l,m,n$) such that
 \begin{align*}
     [F\circ u]_l\leq& C[u]_l\left([F]_1+\|u\|_0^{l-1}[F]_l\right);\\
     [F\circ u]_l\leq& C([F]_1[u]_l+\|DF\|_{l-1}[u]_1^{l-1}).
 \end{align*}
\end{lemma}
Finally, we recall the standard mollification estimates. For $f:\Omega\to \R^m$, let $f\ast\varphi_\ell$ denote the convolution of $f$ with  a standard, radially symmetric mollifier  $\varphi_{\ell}$, i.e. $\varphi_\ell (x) = \ell^{-n}\varphi( x/\ell)$ for some non-negative, radially symmetric $\varphi\in C^\infty_c(B_1)$ with $\int \varphi \,dx = 1$. We then have the following estimates (see \cite{CDS} for a proof).

\begin{lemma}\label{l:mollification}
For any $r, s\geq0,$ and $0<\alpha\leq1,$ we have
\begin{align}
&[f*\varphi_l]_{r+s}\leq Cl^{-s}[f]_r,\\
&\|f-f*\varphi_l\|_r\leq Cl^{2-r}[f]_2, \text{ if } 0\leq r\leq2,\\
&\|(f_1f_2)*\varphi_l-(f_1*\varphi_l)(f_2*\varphi_l)\|_r\leq Cl^{2\alpha-r}\|f_1\|_\alpha\|f_2\|_\alpha,\label{e:quadraticmoll}
\end{align}
with constant $C$ depending only on $n, s, r, \alpha$ and $\varphi.$
\end{lemma}
\end{appendix}

\bibliographystyle{plain}
\bibliography{bib}
\end{document}